\documentclass[11pt]{article}
\usepackage[utf8]{inputenc}
\usepackage[T1]{fontenc}
\usepackage[margin=1in]{geometry}
\usepackage{graphicx}
\usepackage{multirow}
\usepackage{amsmath,amssymb,amsfonts}
\usepackage{amsthm}
\usepackage{mathrsfs}
\usepackage[title]{appendix}
\usepackage{xcolor}
\usepackage{textcomp}
\usepackage{manyfoot}
\usepackage{booktabs}
\usepackage{algorithm}
\usepackage{algorithmicx}
\usepackage{algpseudocode}
\usepackage{listings}
\usepackage{cases}
\usepackage{tikz}
\usepackage{pgfplots}
\pgfplotsset{compat=1.18}
\usepackage{float}
\usepackage{natbib}
\usepackage{url}
\usepackage{microtype}
\usepackage{hyperref}

\newcommand{\h}[1]{\mathbf{#1}}
\newcommand{\argmin}{\operatornamewithlimits{argmin}}
\newcommand{\x}{\mathbf{x}}
\newcommand{\y}{\mathbf{y}}
\newcommand{\z}{\mathbf{z}}

\newcommand{\lipf}{\mathrm{Lip}_{\nabla f}}

\newtheorem{theorem}{Theorem}
\newtheorem{proposition}[theorem]{Proposition}
\newtheorem{lemma}{Lemma}

\newtheorem{remark}{Remark}
\newtheorem{definition}{Definition}

\title{ADMM and Linearized ADMM for Weakly Convex Minimization}
\author{%
Shenghan Mei$^{1}$, Chengyu Ke$^{2}$, Yifei Lou$^{1}$, and Miju Ahn$^{2,*}$\\[0.5em]
\small $^{1}$Department of Mathematics, University of North Carolina at Chapel Hill, Chapel Hill, NC, USA\\
\small $^{2}$Department of Operations Research and Engineering Management, Southern Methodist University, Dallas, TX, USA\\
\small $^{*}$Corresponding author: \texttt{mijua@smu.edu}
}
\date{}

\begin{document}
\maketitle

\begin{abstract}
We study a class of weakly convex optimization problems in which the objective is the sum of a smooth convex term and a weakly convex term that may be nonsmooth. To exploit this structure, we develop a splitting technique based on the alternating direction method of multipliers (ADMM), which decouples the minimization of the two components into tractable subproblems. Because the update associated with the smooth term may require an inner iterative solver, we further linearize this term, yielding a linearized ADMM (LADMM) scheme with an inexpensive one‑step update.
Under mild conditions, we establish the subsequence convergence of both ADMM and LADMM methods to directional stationary solutions, which are equivalent to critical points and Clarke stationary solutions for our weakly convex problem.
Numerical experiments on two low-dimensional test functions and a high-dimensional logarithmic regularized logistic regression model demonstrate that the proposed approaches are computationally efficient and produce solutions of comparable quality to baseline methods.
\end{abstract}


\section{Introduction}

Optimization is a fundamental tool in applied mathematics and scientific computing, with broad applications in machine learning, signal processing, statistics, and operations research. Within the diverse landscape of optimization methods, convex optimization has become particularly prominent and influential. It rests on a mature theoretical foundation with strong guarantees (e.g., any local minimizer is globally optimal) and supports the design of efficient, reliable algorithms that scale to high‑dimensional problems \citep{boyd2004convex,nesterov2004introductory,beck2017first}. 
Consequently, convex optimization has become both a standard modeling approach and a practical computational technique across a wide range of applications.

In practice, however, convexity can be overly restrictive: many modern modeling choices involve objective functions that are nonconvex,   yet still possess structure that algorithms can exploit. Consequently, a variety of structured nonconvex optimization approaches along with convergence analysis have also been developed \citep{gao2023convergence,davis2019stochastic,sun2018alternating}. 
An important class that lies between convex and fully nonconvex functions is that of weakly convex (WC) functions, namely functions that become convex after adding a quadratic term \citep{rockafellar1998variational}. 
On the theoretical side, weak convexity admits a well-developed first-order variational framework, e.g., directional derivatives and generalized subdifferentials, which gives rise to rigorous stationarity concepts and supports convergence analysis for optimization algorithms. 
At the same time, the WC assumption is prevalent in large-scale machine learning, for instance in the analysis of stochastic and proximal-type methods for nonsmooth nonconvex objectives \citep{davis2019stochastic}, as well as in 
min-max formulations 
motivated by modern machine-learning applications \citep{liu2021firstorder}.

In this paper, we study a broad class of composite optimization problems whose objective can be written as the sum of two components: a smooth convex function and a (potentially nonsmooth) WC function.  The smooth convex term typically models data fidelity or physics‑based consistency, while the WC term is used to encode additional modeling structure, such as sparsity, robustness, or other forms of regularization. This composite formulation encompasses many widely used models in statistical learning, inverse problems, and signal/image processing \citep{liu2009largescale,shi2012projection}.
For example, nonlinear regression and classification losses combined with weakly convex penalties fall naturally into this framework. Popular WC regularizers include smoothly clipped absolute deviation (SCAD) \citep{fan2001variable}, the minimax concave penalty (MCP) \citep{zhang2010nearly}, logarithmic (LOG) regularization \citep{ke2021iteratively,ke2024generalized}, and the transformed $\ell_1$ penalty \citep{nikolova2000local,lv2009unified,zhang2017minimization}. %

We employ a variable-splitting method, namely the alternating direction method of multipliers (ADMM) \citep{gabay1976dual,boyd2011distributed},  taking advantage of the two-term structure: the smooth term supplies gradient information, while the WC term preserves useful variational properties. Although the original formulation is unconstrained, introducing an auxiliary variable allows the convex term and the weakly convex term to be handled separately. Majorization-minimization (MM) \citep{hunter2004tutorial}, difference-of-convex algorithm (DCA) \citep{PhamDinhLeThi1997}, and related first-order methods such as proximal gradient are also natural alternatives for unconstrained composite problems; however, these methods typically take a single forward step on the smooth term with a size capped by its Lipschitz constant, which can be conservative when this constant is large or difficult to estimate tightly. ADMM instead solves the smooth-term subproblem to high accuracy at each iteration, which can yield faster practical progress in such settings. 

Despite its flexibility, a practical bottleneck in applying ADMM to our problem setting is that the subproblem associated with the smooth term may not admit a closed-form solution. Consequently, it is often necessary to run an inner iterative solver at each outer iteration, which can dominate the overall computational cost. To mitigate this issue, we further adopt a linearized ADMM (LADMM) \citep{wang2012linearized} scheme, in which the smooth term in the corresponding subproblem is replaced by its first-order approximation.
This modification preserves the splitting structure while yielding an inexpensive closed-form update, thereby significantly reducing the per-iteration cost.

Prior convergence analyses for ADMM were largely limited to convex problem settings, but there has been growing interest in extending such guarantees to nonconvex problems. For example, \citet{hong2016convergence} established convergence of the classical ADMM for certain classes of linearly constrained nonconvex problems with possibly nonsmooth terms, and \citet{li2015global} proved global convergence of splitting methods (including ADMM) for nonconvex composite problems under a semi-algebraic assumption. Analogous guarantees have since been developed for the linearized variant (LADMM): \citet{liu2019linearized} analyzed a linearly constrained nonconvex nonsmooth setting and proved that both the constraint violation and first-order optimality residuals vanish; \citet{yashtini2022convergence} proved global convergence of a proximal linearized ADMM under the Kurdyka--\L ojasiewicz (KL) assumption; and, under the same KL assumption, \citet{sun2018alternating} proposed an ADMM-based method that convexifies the resulting nonconvex subproblems via a linearization technique inspired by the difference-of-convex algorithm (DCA), establishing convergence to a critical point. For problems with nonlinear constraints, \citet{elbourkhissi2025convergence} investigated an inexact, linearized ADMM in which both the smooth objective term and the nonlinear constraints are linearized within the augmented Lagrangian, with one of the resulting subproblems solved only approximately at each iteration.

There has been prior work on ADMM-type convergence analysis for strongly and weakly convex composite models with linear constraints \citep{hong2016convergence,zhang2019fundamental}. In contrast, we focus on an unconstrained formulation and establish convergence to a directional stationary (d-stationary) point (see Definition \ref{def:Dstat}). 
In general nonconvex optimization, d-stationary points can be a stronger notion of optimality than critical points because it requires the directional derivative to be nonnegative in
every direction. For example, it has been shown that d-stationarity implies critical points for certain DC programs, and serves as a necessary condition for local optimality \citep{cui2021modern,pang2017computing}. For the specific weakly convex composite structure studied here, these notions in fact coincide (Section~\ref{sec:wc}); we discuss the implications of this equivalence for our contribution there.

Although the unconstrained setting may appear simpler, it eliminates the consensus structure induced by linear constraints that is heavily exploited in prior analyses such as \citet{hong2016convergence}. As a result, our analysis requires a different sufficient-decrease argument that directly leverages the weak convexity modulus. For  ADMM, we establish a sufficient-decrease property of the augmented Lagrangian under explicit parameter conditions. For LADMM, we further introduce a modified augmented Lagrangian to control the lagged-gradient effect arising from linearization, which is critical for the convergence analysis. These developments lead to explicit parameter conditions under which the sequences generated by both ADMM and LADMM admit convergent subsequences whose limit points are d-stationary solutions.

We further corroborate the theory with three numerical case studies, including smooth low-dimensional test functions and a high-dimensional sparse logistic regression model \citep{liu2009largescale}. 
These experiments illustrate (i) the effectiveness of ADMM/LADMM on the two-term structure, (ii) the practical benefit of linearization in reducing per-iteration cost when inner solves are expensive, and (iii) the resulting trade-off between computational time and solution quality across methods, showing that ADMM/LADMM achieve comparable accuracy with significantly reduced runtime in our test problems.

We summarize the main contributions of this work as follows:
\begin{enumerate}
\item We apply ADMM to the \emph{unconstrained} WC composite problem. Since this eliminates the consensus structure exploited by existing linearly constrained nonconvex ADMM analyses (e.g., \citet{hong2016convergence}), we develop a sufficient-decrease argument that directly leverages the weak convexity modulus.

\item We develop a linearized variant (LADMM) that replaces the smooth-term subproblem with an inexpensive closed-form update. To control the resulting lagged-gradient effect, we introduce a Lyapunov correction sequence that avoids the additional structural assumptions, required by existing linearized/proximal ADMM analyses.

\item We establish subsequential convergence of both ADMM and LADMM to a d-stationary point of the original problem, under explicit, verifiable parameter conditions on $\rho$, avoiding the KL or semi-algebraic assumptions typically required for nonconvex ADMM-type convergence guarantees.

\item We conduct numerical experiments on three case studies, ranging from low-dimensional smooth test functions to a high-dimensional sparse logistic regression model, which confirm the practical efficiency of both methods.
\end{enumerate}

The remainder of the paper is organized as follows. In Section \ref{Formulation}, we introduce the weakly convex setting, presenting the ADMM and LADMM algorithms. Section \ref{Algorithms and Analysis} develops the optimality notions, characterizes weak convexity used in our analysis, and establishes convergence guarantees for both methods. Section \ref{sect:numerical} provides a detailed account of the numerical experiments to validate the performance of ADMM and LADMM. Finally, Section \ref{conclusion} summarizes the paper.

\section{Formulation and algorithms}\label{Formulation}
We consider the problem
\begin{align} \label{prob:fplusg}
    \min \limits_{\x \in \mathbb{R}^n} \ \zeta(\x) \triangleq f(\x) + g(\x),
\end{align}
where $f$ is convex and differentiable and $g$ is a weakly convex function that is not necessarily differentiable. The notion of weak convexity is defined in Definition \ref{def:weaklycvx}.

\begin{definition}[weakly convex function] \label{def:weaklycvx}
    A function $g(\x)$ is weakly convex if there exists a constant $\sigma_{g} > 0$ such that $g(\x) + \frac{\sigma_{g}}{2} \| \x\|_2^2$ is convex. The constant $\sigma_g$ is often referred to as the weakly convex modulus of the function $g$.
\end{definition}

To take advantage of the special structure of the two functions $f$ and $g$, we apply ADMM to solve \eqref{prob:fplusg}, as detailed in Section \ref{sect:WC-ADMM}, and further accelerate it using LADMM in Section \ref{sec:LADMM}.  The convergence analysis for both methods is given in Section \ref{Algorithms and Analysis}.

\subsection{ADMM} \label{sect:WC-ADMM}
We rewrite \eqref{prob:fplusg} into an equivalent constrained formulation,
\begin{align}\label{eq:WC-general}
    \min_{\h x, \h z} \quad  f(\h z) + g(\h x) \quad 
    \text{s.t.} \quad  \h x = \h z,
\end{align}
with its corresponding augmented Lagrangian function given by
\begin{align} \label{eq:aug-Lag}
\mathcal{L}(\h x, \h z; \h v) \triangleq f(\h z) + g(\h x) + \rho \langle \h v, \h x- \h z \rangle + \frac{\rho}{2}\|\h x- \h z\|_2^2, 
\end{align}
where $\h v \in \mathbb{R}^n$ is a dual variable and $\rho>0$ is a parameter.
The ADMM iterations proceed as follows:
\begin{equation} \label{admm}
\left\{\begin{array}{l}
\h x^{t+1} \in \argmin_{\h x} \mathcal{L}(\h x, \h z^t; \h v^{t}) \\
\h z^{t+1} \in \argmin_{\h z} \mathcal{L}(\h x^{t+1}, \h z; \h v^t) \\
\h v^{t+1} = \h v^t + \h x^{t+1} - \h z^{t+1},
\end{array}\right.
\end{equation}
where the superscript $t$ counts the iteration number.
By algebraic manipulation, the $\h x$-subproblem in \eqref{admm} can be expressed equivalently as 
\begin{align} \label{prob:gwidehat}
 \min_{\h x} \quad \widehat{g}(\x) \triangleq g(\h x) + \frac{\rho}{2}\|\h x- \h z^t   + \h v^t\|_2^2.  
\end{align}
Proposition \ref{prop:x_sub_convex} guarantees that \eqref{prob:gwidehat} is a convex problem if the parameter $\rho$ is greater than or equal to the weakly convexity modulus %
of $g$ (defined in Definition~\ref{def:weaklycvx}). 

\begin{algorithm}
\caption{ADMM for weakly convex minimization}
\label{alg:admm}
\begin{algorithmic}[1]
\State Parameters: $\rho \in \mathbb{R}^+$ and tMAX $\in \mathbb{Z}^+$; 
\State Initialize iterates $\h x^0, \h z^0, \h v^0$, and $t=0$; 
\While{$t<\text{tMAX}$} 
\State Compute $\h x^{t+1} \in \argmin_{\h x} g(\h x) + \rho \langle \h v^t, \h x \rangle + \frac{\rho}{2}\|\h x- \h z^t\|_2^2$;
\State Compute $\h z^{t+1} \in \argmin_{\h z} f(\h z) - \rho \langle \h v^t, \h z \rangle + \frac{\rho}{2}\|\h x^{t+1}- \h z\|_2^2$;
\State{$\h v^{t+1} = \h v^t + \h x^{t+1} -\h  z^{t+1}$;} \State{$t=t+1$;}
\EndWhile
\State \Return $\h x^*=\h x^t$
\end{algorithmic}
\end{algorithm}
\begin{proposition} \label{prop:x_sub_convex}
Suppose $g$ is weakly convex with modulus $\sigma_g > 0$. If $\rho \geq \sigma_{g}$, then the problem \eqref{prob:gwidehat} is convex.
\end{proposition}
\begin{proof}
We can rewrite $\widehat{g}$ as follows,
\begin{align*}
    \widehat{g}(\x) &= g(\x) + \frac{\rho}{2} \| \x \|_2^2 + \rho \langle \x, -\z^t + \h v^t \rangle + \frac{\rho}{2} \| -\z^t + \h v^t \|_2^2\\
    &= \Big(g(\x) + \frac{\sigma_g}2\|\h x\|_2^2\Big) + \frac{\rho-\sigma_g}{2} \| \x \|_2^2 + \rho \langle \x, -\z^t + \h v^t \rangle + \frac{\rho}{2} \| -\z^t + \h v^t \|_2^2.
\end{align*}
Since $g$ is weakly convex with modulus $\sigma_g>0$, the first term in the above equation is convex, and hence $\widehat{g}$ is a convex function if $\rho\geq \sigma_g$.
\end{proof}
Likewise, the $\h z$-subproblem in \eqref{admm} can be written as 
\begin{align*}
 \min_{\h z} \quad \hat{f}(\h z) \triangleq f(\h z) + \frac{\rho}{2}\|\h x^{t+1} - \h z + \h v^t\|_2^2,
\end{align*}
which is a convex problem, due to the convexity assumption of $f(\cdot)$.
Consequently, if $\rho \ge \sigma_g$, both the $\h x$- and $\h z$-subproblems in \eqref{admm} are convex, and their global optima can be obtained by standard optimization methods such as gradient descent (GD) \citep{demyanov1978multistep} or Newton's method \citep{boyd2004convex}. 
The overall procedure is summarized in
Algorithm \ref{alg:admm}. Please refer to 
Supplementary Material 
(Appendices A.1, B.1, and C.1) for additional algorithmic details for specific choices of $f$ and $g$.

\subsection{Linearized ADMM}\label{sec:LADMM}

In many applications, the $\h x$-subproblem in \eqref{admm} can be computed in a single step using a proximal operator, whereas the $\h z$-subproblem involves the smooth term $f$ and typically has no closed-form solution; as a result, it may require an inner iterative routine at each outer iteration, which can dominate the overall runtime. 
To avoid nested loops, we consider a linearized ADMM variant that replaces $f(\h z)$ in the $\h z$-subproblem with a first-order approximation
at $\h z^t$, i.e.,
$$
f(\h z)\approx f(\h z^t)+\langle \nabla f(\h z^t),\,\h z-\h z^t\rangle.
$$

LADMM preserves the ADMM splitting structure, so the $\h x$-subproblem and the  dual update remain unchanged. It iterates as follows,
\begin{equation} \label{ladmm}
\left\{\begin{array}{l}
\h x^{t+1} \in \argmin_{\h x} \mathcal{L}(\h x, \h z^t; \h v^{t}) \\
\h z^{t+1} \in \argmin_{\h z} \ \widehat{f}(\h z; \, \h z^t) \\
\h v^{t+1} = \h v^t + \h x^{t+1} - \h z^{t+1},
\end{array}\right.
\end{equation}
where $\widehat{f}(\h z; \, \h z^t) = f(\h z^t) + \langle \nabla f(\h z^t), \h z - \h z^t \rangle + \frac{\rho}{2}\|\h x^{t+1} - \h z + \h v^t\|_2^2$. 
The $\h z$-update from \eqref{ladmm} can be computed in a single closed-form step, 
\begin{align}\label{ladmm_z_sol}
    \h z^{t+1} = \h x^{t+1}+\h v^t-\frac{\nabla f(\h z^t)}{\rho}.
\end{align}

Compared with ADMM, LADMM avoids iterative inner solves for the $\h z$-subproblem, often reducing per-iteration cost while preserving the same $\h x$-subproblem and dual step. The overall procedure is summarized in Algorithm \ref{alg:ladmm}. Detailed updates for the specified forms of $f$ and $g$ are provided in 
Supplementary Material (Appendices A.2, B.2, and C.2). 

\begin{algorithm}
\caption{LADMM for weakly convex minimization}
\label{alg:ladmm}
\begin{algorithmic}[1]
\State Parameters: $\rho \in \mathbb{R}^+$ and tMAX $\in \mathbb{Z}^+$; 
\State Initialize iterates $\h x^0, \h z^0, \h v^0$, and $t=0$; 
\While{$t<\text{tMAX}$} 
\State Compute $\h x^{t+1} \in \argmin_{\h x} g(\h x) + \rho \langle \h v^t, \h x \rangle + \frac{\rho}{2}\|\h x- \h z^t\|_2^2$;
\State Update $\h z^{t+1}$ via \eqref{ladmm_z_sol};
\State{$\h v^{t+1} = \h v^t + \h x^{t+1} -\h  z^{t+1}$;} \State{$t=t+1$;}
\EndWhile
\State \Return $\h x^*=\h x^t$
\end{algorithmic}
\end{algorithm}

\section{Theoretical analysis}\label{Algorithms and Analysis}

This section is devoted to the convergence analysis of ADMM and LADMM.
Specifically, in Section \ref{sec:wc}, we characterize weak convexity of a function and review the notion of a stationary solution, which forms the foundation of our subsequent analysis.
With these tools in place, we establish convergence guarantees for ADMM in Section~\ref{admm: conv} and extend the analysis to the LADMM scheme in Section \ref{sect:DC-ADMM}.

\subsection{Characterizing weak convexity and stationarity} \label{sec:wc} 

In the literature, there are a number of stationary solutions %
used for analyzing nonconvex and nondifferentiable problems. It turns out some of widely used stationary solutions are equivalent for the problem \eqref{prob:fplusg}. This can be shown by rewriting \eqref{prob:fplusg} as a difference-of-convex (DC) function (see Definition \ref{def:DC}, then apply the results from the literature \citep{pang2017computing,joki2018double,cui2021modern}. To do so, we introduce a DC function in Definition~\ref{def:DC}.

\begin{definition}[DC function/program] \label{def:DC}
    A function is called difference-of-convex (DC) if one can rewrite the function as a difference of two convex functions. A DC program is defined by
\begin{align} \label{prob:plainDC}
   \min_{\x\in \mathbb{R}^n} \theta(\x) \triangleq \phi(\x) - \psi(\x), 
\end{align}
where both functions $\phi$ and $\psi$ are convex. 
\end{definition}

It is straightforward that weakly convex functions are a special case of DC functions: if $g$ is weakly convex with modulus $\sigma_g$, then
\begin{align} \label{prob:trivialDC}
    \zeta(\x) = \underbrace{\zeta(\x) + \frac{\sigma_{g}}{2} \| \x \|_2^2}_{\rm{convex}} - \frac{\sigma_{g}}{2} \| \x \|_2^2,
\end{align}
so the objective function $\zeta$ admits a DC decomposition. Since the problem \eqref{prob:fplusg} has a trivial DC representation \eqref{prob:trivialDC} with a differentiable concave component, 
one can apply \citet[Proposition~6.1.10]{cui2021modern}, which establishes the equivalence among directional stationary solutions, critical points, and Clarke stationary solutions of \eqref{prob:trivialDC}. We refer to \citet{cui2021modern} for a detailed discussion of these notions. These stationary concepts, defined respectively via the subdifferential, directional derivatives, and the Clarke subdifferential \citep{bazaraa2006nonlinear,clarke1983optimization}, are standard in the DC programming literature. Leveraging their equivalence, we adopt the directional stationary solution for our analysis, as formalized in Definition \ref{def:Dstat}. Consequently, proving that ADMM or LADMM produces a directional stationary point immediately establishes the other two stationarity properties.

We provide an equivalent characterization of weakly convex functions (see Definition~\ref{def:weaklycvx}). As shown in Proposition~\ref{prop:wc}, this result is a specialization of \cite[Proposition~4.8]{vial1983strong} and is stated using the directional derivative (Definition \ref{def:d-derivative}) so that the existing result can be applied directly in our analysis.

\begin{definition}[Directional derivative]\label{def:d-derivative}
The directional derivative of a function $g$ at a point $\x$ in the direction $\h d$, denoted $g'(\x; \h d)$, is defined as 
\begin{align*}
    g' (\x; \h d) = \lim \limits_{\tau \downarrow 0} \frac{g(\x + \tau \h d) - g(\x)}{\tau},
\end{align*}
if the limit exists. In this case, the function $g(x)$ is directionally differentiable.
\end{definition}

\begin{definition}[directional stationary point] \label{def:Dstat}
    A point $\x^*\in \mathbb{R}^n$ is a directional-stationary solution (or d-stationary for short)  of \eqref{prob:fplusg} if $\zeta ' (\x^*; \x - \x^*) \geq 0$ for all $\x\in \mathbb{R}^n$.
\end{definition}

For any directionally differentiable function, we present an equivalence condition of the weakly convexity of the function. We include a proof to make the paper self-contained.

\begin{proposition} \label{prop:wc}
     The function $g$ is weakly convex if and only if $g$ is directionally differentiable and there exists $\sigma_g > 0$ for which
    \begin{align} \label{ineq:wc_dderiv}
    - g(\x) + g(\y) + g'(\y; \x - \y)  \leq \frac{\sigma_g}{2} \|\x - \y \|_2^2, \text{ for all } \x, \y \in \mathbb{R}^n. 
    \end{align}
\end{proposition}
\begin{proof}
If $g$ is weakly convex, then it is DC and therefore directionally differentiable, since the directional derivative of a DC function equals the difference of the directional derivatives of its convex components \citep{bazaraa2006nonlinear}. 
By applying the first-order condition of a convex function to $g(\x) + \frac{\sigma_g}{2} \| \x \|_2^2$, we have 
    \begin{align*}
        g(\y) + \frac{\sigma_g}{2} \| \y \|_2^2 + g'(\y; \x - \y) + \sigma_g \langle \y, \x - \y \rangle \leq g(\x) + \frac{\sigma_g}{2} \| \x \|_2^2 \quad \forall \x, \y,
    \end{align*}
    which is equivalent to \eqref{ineq:wc_dderiv}. %

For the other direction, we assume $g$ is directionally differentiable and there exists $\sigma_g>0$ such that
\eqref{ineq:wc_dderiv} holds. Define $\tilde g(\x)=g(\x)+\frac{\sigma_g}{2}\|\x\|_2^2$.
Since $\frac{\sigma_g}{2}\|\x\|_2^2$ is differentiable, $\tilde g$ is directionally differentiable and
$$
\tilde g'(\y;\x-\y)=g'(\y;\x-\y)+\sigma_g\langle \y,\x-\y\rangle.
$$
Rearranging \eqref{ineq:wc_dderiv} gives
$$
g(\x)\ge g(\y)+g'(\y;\x-\y)-\frac{\sigma_g}{2}\|\x-\y\|_2^2.
$$
Adding $\frac{\sigma_g}{2}\|\x\|_2^2$ to both sides and using
$\|\x\|_2^2-\|\x-\y\|_2^2=\|\y\|_2^2+2\langle \y,\x-\y\rangle$ yields
$$
\tilde g(\x)\ge \tilde g(\y)+\tilde g'(\y;\x-\y),\quad \forall \x,\y.
$$
By the standard directional-derivative characterization of convexity, $\tilde g$ is convex.
Therefore $g$ is weakly convex with modulus $\sigma_g$.
\end{proof}



Before proceeding, we highlight the technical features that distinguish our analysis from existing convergence results for nonconvex ADMM and DC programming methods. First, the weakly convex structure of $g$ makes the $\h x$-subproblem nonconvex unless the penalty parameter satisfies $\rho \geq \sigma_g$; this parameter condition is essential and enters explicitly in our sufficient-decrease argument (Lemma~\ref{lem:suff_decrease}). Our splitting constraint $\h x = \h z$ in \eqref{eq:WC-general} is a special case of the general linear constraint studied for nonconvex ADMM by \citet{hong2016convergence}; however, their convergence theory requires the nonconvex term to be differentiable with a Lipschitz-continuous gradient, an assumption our possibly nonsmooth weakly convex $g$ need not satisfy. Consequently, introducing the splitting variable in \eqref{eq:WC-general} does not by itself bring problem \eqref{prob:fplusg} within the scope of \citet{hong2016convergence}; our sufficient-decrease argument instead directly exploits the weak convexity modulus $\sigma_g$.

Second, for LADMM, linearizing $f$ in the $\h z$-subproblem introduces a lagged-gradient term that prevents the augmented Lagrangian from decreasing monotonically. To address this, we introduce the auxiliary sequence $\{\kappa_t\}$ (defined in \eqref{def:kappa}), which absorbs the lagged-gradient effect and restores a sufficient-decrease property (Lemma~\ref{lem:kappa_conv}). This differs in both mechanism and assumptions from existing linearized and proximal ADMM analyses: \citet{sun2018alternating} convexifies the nonconvex subproblem via a DCA-inspired linearization, but requires the concave component of the DC decomposition to be Lipschitz, which our weakly convex $g$ need not satisfy; \citet{yashtini2022convergence} establishes global convergence of a proximal linearized ADMM under the Kurdyka--\L ojasiewicz (KL) property, but requires the nonsmooth term to decompose as the sum of a proper lower-semicontinuous function and a coercive function, a structural assumption we do not impose.

Third, although directional stationarity coincides with critical-point and Clarke stationarity for DC problem  (Proposition~6.1.10, \citet{cui2021modern}), our analysis reaches this common set of points directly via directional derivatives, under the explicit parameter conditions above, rather than through the Kurdyka--\L ojasiewicz or semi-algebraic assumptions required by \citet{li2015global}, \citet{sun2018alternating}, and \citet{yashtini2022convergence} to obtain their critical-point guarantees.

\subsection{Convergence of ADMM}\label{admm: conv}
We establish the convergence of ADMM in \eqref{admm} based on the following assumptions about the objective function in \eqref{prob:fplusg}, 
\begin{enumerate}
\renewcommand{\theenumi}{A\arabic{enumi}}
\renewcommand{\labelenumi}{(\theenumi)}
    \item The function $\zeta$ is coercive, i.e., $\zeta(\h x) \to \infty$ as $\|\h x\|_2 \to \infty$; \label{assume:coercive}
    
    \item The function $\zeta$ is lower-bounded;  \label{assume:lb}
    
    \item The function $g$ is proper, lower semicontinuous, finite-valued, weakly convex with modulus $\sigma_g$, and directionally differentiable. \label{assume:ggrad_lip}

    \item \label{assume:fgrad_lip} The function $f$ has Lipschitz continuous gradient, i.e., there exists a constant $\lipf > 0$ such that
    \begin{align} 
    \| \nabla f(\x)- \nabla f(\y) \|_2 \leq \lipf \| \x - \y \|_2, \quad \forall \, \x, \y. \label{ineq:fgrad_lip}
    \end{align} 
\end{enumerate}

The roles of these assumptions are as follows. The weak convexity of $g$ and the
Lipschitz continuity of $\nabla f$ are used to establish the sufficient-decrease estimates
for the augmented Lagrangian in ADMM. The lower boundedness assumption ensures that the corresponding
sequence $\{\mathcal{L}(\h x^t,\h z^t;\h v^t)\}_{t = 0}^\infty$ is lower bounded. The coercivity of $\zeta$ is used to obtain boundedness
of the iterates $\{(\h x^t, \h z^t, \h v^t)\}_{t=0}^\infty$. Together, these descent and boundedness properties allow us to
extract convergent subsequences and show that the successive differences vanish.
Finally, the properness, lower semicontinuity, finite-valuedness, and directional differentiability of $g$
ensure that the directional derivatives and limiting arguments used in the stationarity
proof are well defined.

We show in Lemma \ref{lem:suff_decrease} that the augmented Lagrangian function \eqref{eq:aug-Lag} decreases sufficiently at each ADMM iteration under certain conditions, followed by Theorem \ref{thm:admm_convergence} for the convergence of the iterations to a d-stationary solution of the problem \eqref{prob:fplusg}. The key idea is that each primal update yields a descent controlled by the successive differences $\|\h x^{t+1} - \h x^t\|_2^2$ and $\|\h z^{t+1} - \h z^t\|_2^2$; telescoping then forces these differences to vanish.

\begin{lemma} \label{lem:suff_decrease}
(Sufficient decrease)   If $\rho > \max \{\sigma_{g} , \sqrt{2} \lipf \}$ and Assumptions \ref{assume:ggrad_lip}--\ref{assume:fgrad_lip} hold, then each ADMM iteration \eqref{admm} satisfies
\begin{align} \label{lem:suff_decrease_ineq}
\begin{aligned}
&\mathcal{L}(\h x^{t+1}, \h z^{t+1}; \h v^{t+1}) - \mathcal{L}(\h x^t, \h z^t; \h v^t) \\
&\hspace{5pc} \leq \left( \frac{\sigma_{g} -\rho}{2} \right) \|\h x^{t+1} -\h x^t\|_2^2 + \left(\frac{\lipf^2}{\rho} -\frac{\rho}{2}\right)\|\h z^{t+1} - \h z^t\|_2^2.
\end{aligned}
\end{align} 
\end{lemma}
\begin{proof}
Since $\widehat g$ is convex by Proposition \ref{prop:x_sub_convex} and $\h x^{t+1}\in\argmin_\h x \widehat g(\h x)$ by the ADMM scheme \eqref{admm}, we have
$\h 0\in \partial \widehat g(\h x^{t+1})$, which implies that $\h x^{t+1}$ is a d-stationary point of $\widehat{g}$, i.e., $\widehat g'(\h x^{t+1};\h x^t-\h x^{t+1})\ge 0$. Therefore, we can get
\begin{align*}
    \mathcal{L}(\h x^{t+1}, \h z^t; \h v^t) - \mathcal{L}(\h x^t,\h z^t ; \h v^t) 
    &= \widehat{g}(\h x^{t+1}) - \widehat{g}(\h x^t) \\
    &\leq \widehat{g}(\h x^{t+1}) - \widehat{g}(\h x^t) + \widehat{g}'(\h x^{t+1}; \x^{t} - \x^{t+1}).
\end{align*}
Using the definition of $\hat g$ in \eqref{prob:gwidehat}, we have
\begin{align*}
&\widehat{g}(\h x^{t+1}) - \widehat{g}(\h x^t) + \widehat{g}'(\h x^{t+1}; \x^{t} - \x^{t+1})\\
    =& g(\x^{t+1}) + \frac{\rho}{2} \| \x^{t+1} - \z^t + \h v^t \|_2^2 
    - g(\x^{t}) - \frac{\rho}{2} \| \x^{t} - \z^t + \h v^t \|_2^2 \\
    &\qquad + g'(\h x^{t+1}; \x^{t} - \x^{t+1}) + \rho \langle \x^{t+1} - \z^t + \h v^t,\  \x^t - \x^{t+1} \rangle  \\
    =& \left\{ - g(\x^{t}) + g(\x^{t+1}) + g'(\h x^{t+1}; \x^{t} - \x^{t+1}) \right\} \\
    &\qquad + \frac{\rho}{2} \left\{  \| \x^{t+1} - \z^t + \h v^t \|_2^2 - \| \x^{t} - \z^t + \h v^t \|_2^2 + 2 \langle \x^{t+1} - \z^t + \h v^t,\  \x^t - \x^{t+1} \rangle\right\} \\[0.5pc]
    =& \left\{ - g(\x^{t}) + g(\x^{t+1}) + g'(\h x^{t+1}; \x^{t} - \x^{t+1}) \right\} - \frac{\rho}{2} \| \x^t - \x^{t+1} \|_2^2. 
\end{align*}
It further follows from Assumption~\ref{assume:ggrad_lip} that
\begin{align} \label{x_dec}
    \mathcal{L}(\h x^{t+1}, \h z^t; \h v^t) - \mathcal{L}(\h x^t,\h z^t ; \h v^t) 
    &\leq \frac{(\sigma_g - \rho)}{2} \| \x^t - \x^{t+1} \|_2^2. 
\end{align}
For the $\h z$-subproblem, the optimality of $\z^{t+1}$ yields 
\begin{align} \label{equality:z_opt}
\nabla f(\h z^{t+1}) - \rho ( \h x^{t+1} - \h z^{t+1} + \h v^t) = \h 0,
\end{align}
which leads to the following chain of inequalities 
\begin{align}
    &\mathcal{L}(\h x^{t+1}, \h z^{t+1}; \h v^t) - \mathcal{L}(\h x^{t+1}, \h z^t; \h v^t) \nonumber \\
    =& f(\h z^{t+1}) - f(\h z^t) +\frac{\rho}{2} \Big( \, \|\h x^{t+1} - \h z^{t+1} + \h v^t\|_2^2 -  \|\h x^{t+1} - \h z^t + \h v^t\|_2^2 \, \Big) \nonumber \\
    \leq & - \langle \nabla f(\h z^{t+1}),  \h z^{t} - \h z^{t+1} \rangle +\frac{\rho}{2}  \langle \, \z^{t+1} + \z^{t} - 2(\x^{t+1} + \h v^{t}), \z^{t+1} - \z^{t} \, \rangle \text{ by convexity of } f \nonumber\\
    =& \rho \langle \x^{t+1} - \z^{t+1} + \h v^{t}, \z^{t+1} - \z^{t} \rangle +\frac{\rho}{2}  \langle \, \z^{t+1} + \z^{t} - 2(\x^{t+1} + \h v^{t}), \z^{t+1} - \z^{t} \, \rangle \text{ by } \eqref{equality:z_opt} \nonumber\\
    =& - \frac{\rho}{2} \| \z^t - \z^{t+1} \|_2^2. \label{ineq:z-sub}
\end{align}

Lastly, using the $\h v$-update \eqref{admm} and the optimality of $\z^{t+1}$ in \eqref{equality:z_opt}, we obtain
\begin{align} \label{eq:v_update}
\h v^{t+1}  = \h v^t +\h x^{t+1} - \h z^{t+1} = \frac{\nabla f(\h z^{t+1})}{\rho}, 
\end{align}
which then yields
\begin{align} \label{v_dec}
&\mathcal{L}(\h x^{t+1}, \h z^{t+1}; \h v^{t+1}) - \mathcal{L}(\h x^{t+1}, \h z^{t+1}; \h v^t) \nonumber \\
=& \rho \langle \h v^{t+1} - \h v^{t}, \h x^{t+1}-\h z^{t+1}\rangle \nonumber =\rho\|\h v^{t+1} - \h v^t\|_2^2 \nonumber \\
=&\frac{1}{\rho}\|\nabla f(\h z^{t+1})-\nabla f(\h z^t)\|_2^2 \leq \frac{\lipf^2}{\rho}\|\h z^{t+1} - \h z^t\|_2^2.
\end{align}
The last inequality is obtained by Assumption \ref{assume:fgrad_lip} where the Lipschitz constant of the gradient of $f$ is defined in \eqref{ineq:fgrad_lip}. 
Combining the inequalities \eqref{x_dec}, \eqref{ineq:z-sub} and \eqref{v_dec}, we complete the proof.
\end{proof}

\begin{theorem} \label{thm:admm_convergence}
  If $\rho > \max\{\sigma_g,\sqrt{2}\lipf\}$ and Assumptions \ref{assume:coercive}--\ref{assume:fgrad_lip} hold, 
  then the iterates $\{(\x^{t}, \z^{t}, \h v^{t})\}_{t=1}^{\infty}$ generated by ADMM in \eqref{admm} has a subsequence convergent to a $d$-stationary point of problem \eqref{prob:fplusg}. 
\end{theorem}
\begin{proof}
Using the assumption that $\rho > \max\{\sigma_g, \sqrt{2}\lipf\}$, we have $ \frac{\sigma_g -\rho}{2} < 0$ and $\frac{\lipf^2}{\rho} -\frac{\rho}{2} < 0$. The telescoping summation of inequality \eqref{lem:suff_decrease_ineq} from $t=0$ to $T$ for an arbitrary integer $T = 0,1,\cdots$ leads to
\begin{equation} \label{ineq:telescope}
\begin{aligned}
&\mathcal{L}(\h x^{T+1},\h z^{T+1};\h v^{T+1}) -\mathcal{L}(\h x^0,\h z^0;\h v^0)   \\
&\hspace{1pc} \leq
\left( \frac{\sigma_g-\rho}{2} \right) \sum_{t=0}^T\|\h x^{t+1}-\h x^t\|_2^2  
+ \left(\frac{\lipf^2}{\rho} -\frac{\rho}{2}\right) \sum_{t=0}^T\|\h z^{t+1}-\h z^t\|_2^2 \le 0, 
\end{aligned}
\end{equation}
which means that the sequence $\{\mathcal{L}(\h x^t,\h z^t;\h v^t)\}_{t = 0}^\infty$  is upper-bounded by $\mathcal{L}(\h x^0,\h z^0;\h v^0)$.

Additionally, using Assumptions \ref{assume:lb} and \ref{assume:fgrad_lip} together with \eqref{eq:v_update}, we obtain a lower bound 
of $\mathcal L(\h x^{T+1}, \h z^{T+1}; \h v^{T+1}) $, i.e.,
\begin{align} \label{ineq:upperL}
&\mathcal L(\h x^{T+1}, \h z^{T+1}; \h v^{T+1}) \notag \\
=& \displaystyle  f(\h z^{T+1}) - f(\h x^{T+1})+f(\h x^{T+1}) + g(\h x^{T+1}) + \rho\langle\h v^{T+1}, \h x^{T+1}-\h z^{T+1}\rangle  + \frac \rho 2 \|\h x^{T+1}-\h z^{T+1}\|_2^2\notag  \\
\ge & f(\h x^{T+1}) + g(\h x^{T+1}) + \langle\rho\h v^{T+1}-\nabla f(\h z^{T+1}), \h x^{T+1}-\h z^{T+1}\rangle + \frac {\rho-\lipf} 2 \|\h x^{T+1}-\h z^{T+1}\|_2^2  \notag \\
=& 	 f(\h x^{T+1}) + g(\h x^{T+1})  + \frac {\rho-\lipf} 2 \|\h x^{T+1}-\h z^{T+1}\|_2^2,
\end{align}
which also implies that the sequence $\{f(\h x^t) + g(\h x^t)\}_{t = 0}^\infty$ is upper-bounded  if 
$\rho>\lipf$. 
Assumption \ref{assume:coercive} regarding the coerciveness of $\zeta(\h x)$ guarantees that $\{\h x^t\}_{t = 1}^\infty$ is bounded, so is $\{\h z^t\}_{t = 1}^\infty$ by \eqref{ineq:upperL}. To show the boundedness of $\{\h v^t\}_{t = 1}^\infty$, we use \eqref{eq:v_update} to obtain
\[
\|\h v^{T+1} - \h v^1\|_2 \le \frac{\lipf}{\rho}\|\h z^{T+1} - \h z^1\|_2, 
\] 
thus leading to
\begin{equation}\label{ineq:ch4_v_upper}
   \|\h v^{T+1}\|_2\le \|\h v^1\|_2 + \frac{\lipf}{\rho} (\|\h z^{T+1}\|_2+\|\h z^1\|_2),  \; \forall \, T = 0,1,\cdots
\end{equation}
The boundedness of  $\{\h z^t\}_{t = 1}^\infty$ guarantees the boundedness of $\{\h v^t\}_{t = 1}^\infty$. 
Consequently, we show that the sequence $\{(\h x^t, \h z^t, \h v^t)\}_{t=1}^\infty$ is bounded. By the Bolzano-Weierstrass Theorem, it
therefore admits a convergent subsequence, denoted by $(\h x^{t_j}, \h z^{t_j}, \h v^{t_j}) \rightarrow (\h x^*, \h z^*, \h v^*)$ as $t_j \rightarrow \infty$.

$\mathcal{L}(\h x^t,\h z^t;\h v^t)$ is monotonically decreasing and lower bounded due to Lemma~\ref{lem:suff_decrease} and Assumption \ref{assume:lb}.
By letting $T\rightarrow \infty$ in  \eqref{ineq:telescope}, we obtain that 
 $\sum_{t = 0}^\infty\|\h x^{t+1}-\h x^t\|_2^2$ and $\sum_{t = 0}^\infty \|\h z^{t+1}-\h z^t\|_2^2$ are finite, and hence
 $\h x^{t+1}-\h x^t \to 0$ and $\h z^{t+1}-\h z^t \to 0$ as $t \to \infty$.
By \eqref{v_dec}, each term $\|\h v^{t+1} - \h v^t \|_2^2$ is bounded by $\|\h z^{t+1} - \h z^t \|_2^2$, which implies that $\h v^{t+1}-\h v^t \to 0$. Since $(\h x^{t_j}, \h z^{t_j}, \h v^{t_j}) \rightarrow (\h x^*, \h z^*, \h v^*)$, we have $(\h x^{t_j+1}, \h z^{t_j+1}, \h v^{t_j+1}) \rightarrow (\h x^*, \h z^*, \h v^*)$ and $\h x^*=\h z^*$ due to the $\h v$-update.

It remains to show $(\h x^*, \h z^*, \h v^*)$ is a d-stationary point of the problem \eqref{prob:fplusg}. 
    By the iterative scheme  \eqref{admm}, we have
    \begin{align*}
    &\mathcal{L}(\h x^{t_j+1},\h z^{t_j},\h v^{t_j}) \le \mathcal{L}(\h x,\h z^{t_j},\h v^{t_j}),  \quad \forall \h x, \\
    &\mathcal{L}(\h x^{t_j+1},\h z^{t_j+1},\h v^{t_j}) \le \mathcal{L}(\h x^{t_j+1},\h z,\h v^{t_j}),  \quad \forall \h z.
    \end{align*}
    Letting $t_j \to \infty$, we have $\mathcal{L}(\h x^*,\h z^*,\h v^*) \le \mathcal{L}(\h x,\h z^*,\h v^*)$ for all $\h x$ and $\mathcal{L}(\h x^*,\h z^*,\h v^*) \le \mathcal{L}(\h x^*,\h z,\h v^*)$ for all $\h z$, thus leading to
    \begin{align} \label{ineq:L_x1}
     g(\h x^*) + \frac{\rho}{2}\|\h v^*\|_2^2    \le  g(\h x) + \frac{\rho}{2}\| \h x - \h z^* + \h v^* \|_2^2,\quad \forall \h x,
    \end{align}
    and
    \begin{align} \label{ineq:L_z1}
    f(\h z^*) + \frac{\rho}{2}\|\h v^*\|_2^2 \le f(\h z) + \frac{\rho}{2}\|\h x^* - \h z + \h v^* \|_2^2, \quad \forall \h z.
    \end{align}
   We fix a point $\h x$ and choose $\h z = \h x$ in \eqref{ineq:L_z1}. Combining \eqref{ineq:L_x1} and \eqref{ineq:L_z1} yields
    \begin{align}\label{ineq:comp_objval}
    &f(\h x^*) + g(\h x^*) + \rho\|\h v^*\|_2^2
    \le 
    f(\h x) + g(\h x) + \frac \rho 2\left(\|\h x-\h x^*+ \h v^*\|_2^2+\|\h x^*-\h x+ \h v^*\|_2^2\right), 
    \end{align}
    where we use $\h x^*=\h z^*$.
    Let us define  $\widehat{\zeta}(\h x) \triangleq f(\h x) + g(\h x) + \rho\|\h x - \h x^*\|_2^2$. It follows from \eqref{ineq:comp_objval} that 
    \begin{align*}
    \widehat{\zeta}(\h x^*) = f(\h x^*) + g(\h x^*) \le f(\h x) + g(\h x) + \rho \|\h x - \h x^*\|_2^2 = \widehat{\zeta}(\h x), \, \forall \h x,
    \end{align*}
    implying $\h x^*$ is a global minimum 
    for $\widehat{\zeta}$ and hence a d-stationary solution. We show functions $\zeta$ and $\widehat{\zeta}$ have the same directional derivative
    at the point $\h x = \h x^*$ by the following calculations:
     \begin{align*}
    \widehat{\zeta}^{\, \prime}(\h x^*; \h d)  
    & = \lim_{h\rightarrow 0^+}\dfrac{\widehat{\zeta}(\h x^* + h \h d)-\widehat{\zeta}(\h x^*) }{h} \\
        & = \lim_{h\rightarrow 0^+}\dfrac{\zeta(\h x^* + h \h d) + \rho \|\h x^* + h\h d - \h x^*\|_2^2}{h} -\lim_{h\rightarrow 0^+} \frac{\zeta(\h x^*)+ \rho\|\h x^* - \h x^*\|_2^2}{h}
        \\
        & = \lim_{h\rightarrow 0^+}\dfrac{\zeta(\h x^* + h \h d) -\zeta(\h x^*)  + \rho h^2 \|\h d\|_2^2}{h}
        \\
        & = \lim_{h\rightarrow 0^+}\dfrac{\zeta(\h x^* + h \h d) -\zeta(\h x^*) }{h} \\
        & = \zeta'(\h x^*; \h d), \quad \forall \, \h d.
    \end{align*}
    Since $\h x^*$ is a d-stationary solution of 
    $\widehat{\zeta}(\h x)$, we have $0 \leq \widehat{\zeta}'(\x^*; \h d) = \zeta'(\x^*; \h d),$ which concludes the proof. 
\end{proof}

\begin{remark}
The convergence analysis above for ADMM, and the analysis for LADMM to follow, both assume that the respective subproblems are solved exactly. In practice, however, some subproblems are solved via iterative inner solvers. Extending the analysis to inexact updates would require additional assumptions on the stopping criteria, and we leave this as an important direction for future work.
\end{remark}

\subsection{Convergence of Linearized ADMM}\label{sect:DC-ADMM}

In this section, we establish the convergence of LADMM in \eqref{ladmm} under the same Assumptions
\ref{assume:coercive}--\ref{assume:fgrad_lip} as those used for ADMM.
The main technical challenge is that linearizing $f$ in the $\h z$-update
breaks the monotonicity of the augmented Lagrangian
$\mathcal{L}(\h x^t,\h z^t;\h v^t)$, as evident in
Lemma~\ref{lem:suff_decrease_ladmm}.
To address this issue, we introduce the sequence
$\{\kappa_t\}_{t=1}^\infty$ in \eqref{def:kappa}, obtained by augmenting $\mathcal{L}$
with an additional quadratic term.
While linearized ADMM has been studied for nonconvex problems, the present
analysis is tailored to weakly convex minimization, where the nonconvexity of
$g$ is quantified by the weak convexity modulus $\sigma_g$.
Building on  Lemma~\ref{lem:suff_decrease_ladmm},  we further introduce Lemma~\ref{lem:kappa_conv} to prove that $\kappa_t$ admits a sufficient decrease at each iteration,
which implies boundedness of the iterates and vanishing successive differences.
 Theorem~\ref{thm:linearADMM_conv} establishes that the LADMM iterates admit a subsequence
converging to a d-stationary point of \eqref{prob:fplusg} under a suitably stronger condition on $\rho$ compared to the requirement in ADMM.

\begin{lemma} \label{lem:suff_decrease_ladmm}
Under Assumptions \ref{assume:ggrad_lip}--\ref{assume:fgrad_lip}, each LADMM iteration \eqref{ladmm} satisfies
\begin{align} \label{lem:suff_decrease_ineq_ladmm}
\begin{aligned}
&\mathcal{L}(\h x^{t+1}, \h z^{t+1}; \h v^{t+1}) - \mathcal{L}(\h x^t, \h z^t; \h v^t) \\
&\hspace{5pc} \leq \left( \frac{\sigma_{g} -\rho}{2} \right) \|\h x^{t+1} -\h x^t\|_2^2 + \left( \frac{\lipf - \rho}{2} \right) \|\h z^{t+1} - \h z^t\|_2^2 + \frac{\lipf^2}{\rho} \|\h z^{t} - \h z^{t-1} \|_2^2.
\end{aligned}
\end{align} 
\end{lemma}
\begin{proof}
We start with the $\h z$-subproblem in \eqref{ladmm} by analyzing the following difference 
\begin{align}
    &\mathcal{L}(\h x^{t+1}, \h z^{t+1}; \h v^t) - \mathcal{L}(\h x^{t+1}, \h z^t; \h v^t) \nonumber \\
    =& f(\h z^{t+1}) - f(\h z^t) +\frac{\rho}{2} \Big( \, \|\h x^{t+1} - \h z^{t+1} + \h v^t\|_2^2 -  \|\h x^{t+1} - \h z^t + \h v^t\|_2^2 \, \Big)\nonumber\\
	\leq &  \langle \nabla f(\h z^{t}),  \h z^{t+1} - \h z^{t} \rangle
	+ \displaystyle{\frac{\lipf}{2}} \| \h z^{t+1} - \h z^{t} \|_2^2 +\frac{\rho}{2} \Big( \, \|\h x^{t+1} - \h z^{t+1} + \h v^t\|_2^2 -  \|\h x^{t+1} - \h z^t + \h v^t\|_2^2 \, \Big)  \nonumber \\
    = & \widehat{f}(\h z^{t+1}; \h z^t) - 	\widehat{f}(\h z^{t}; \h z^t) + \displaystyle{\frac{\lipf}{2}} \| \h z^{t+1} - \h z^{t} \|_2^2  , \label{ineq:ladmm-z}
    \end{align}
where we use the definition of $\widehat{f}$ and the Lipschitz Assumption \ref{assume:fgrad_lip} that leads to 
\begin{align} \label{ineq_flip}
f(\h z^{t+1}) \leq f(\h z^{t}) + \langle \nabla f(\h z^{t}), \h z^{t+1} - \h z^{t} \rangle + \displaystyle{\frac{\lipf}{2}} \| \h z^{t+1} - \h z^{t} \|_2^2.
\end{align}
Using the strong convexity of $\widehat{f}(\h z; \, \h z^t)$ with respect to $\h z$ and the optimality condition $\nabla \widehat{f}(\h z^{t+1};\,\h z^t)=\h 0$, we have 
\begin{align} \label{ineq_fhat_stronglyconvx}
\widehat{f}(\h z^t; \, \h z^t) 
&\geq \widehat{f}(\h z^{t+1}; \, \h z^t) + \langle \nabla \widehat{f}(\h z^{t+1}; \, \h z^t), \h z^t - \h z^{t+1} \rangle + \displaystyle{\frac{\rho}{2}} \|  \h z^t - \h z^{t+1} \|_2^2 \nonumber \\
&= \widehat{f}(\h z^{t+1}; \, \h z^t) + \displaystyle{\frac{\rho}{2}} \|  \h z^t - \h z^{t+1} \|_2^2.
\end{align}
It follows from \eqref{ineq:ladmm-z} and \eqref{ineq_fhat_stronglyconvx} that
\begin{equation}\label{ineq:z-sub-ladmm}
    \mathcal{L}(\h x^{t+1}, \h z^{t+1}; \h v^t) - \mathcal{L}(\h x^{t+1}, \h z^t; \h v^t) \leq \frac{\lipf - \rho}{2} \| \z^{t+1} - \z^t \|_2^2.
\end{equation}

Lastly, using the optimality condition of the $\h z$-subproblem
\begin{align} \label{equality:z_opt_ladmm}
\nabla f(\h z^{t}) - \rho ( \h x^{t+1} - \h z^{t+1} + \h v^t) = \h 0,
\end{align}
we have 
\begin{equation}\label{eq:v_update_ladmm}
\h v^{t+1}=\h v^{t}+\h x^{t+1}-\h z^{t+1}=\frac{\nabla f(\h z^{t})}{\rho},
\end{equation}
from which the last term in \eqref{lem:suff_decrease_ineq_ladmm} follows by applying the following inequality
\begin{align} \label{v_dec_ladmm}
\mathcal{L}(\h x^{t+1}, \h z^{t+1}; \h v^{t+1}) - \mathcal{L}(\h x^{t+1}, \h z^{t+1}; \h v^t) \leq \frac{\lipf^2}{\rho}\|\h z^{t} - \h z^{t-1}\|_2^2.
\end{align}
Combining \eqref{x_dec}, \eqref{ineq:z-sub-ladmm}, and \eqref{v_dec_ladmm} completes the proof.
\end{proof}

Unlike ADMM, the descent argument in Lemma~\ref{lem:suff_decrease} cannot be applied directly to LADMM, 
because the $\h v$-update in LADMM involves $\nabla f(\h z^{t})$ in \eqref{eq:v_update_ladmm} rather than $\nabla f(\h z^{t+1})$ in \eqref{eq:v_update}. As a result, an additional lagged-gradient term appears, and the augmented Lagrangian
$\mathcal{L}(\h x^t,\h z^t,\h v^t)$ alone no longer yields a sufficient decrease.
To address this issue, we introduce an auxiliary function:
\begin{equation}\label{def:kappa}
\kappa_t \triangleq \mathcal{L}(\h x^t, \h z^t, \h v^t) + \lipf^2 \| \h z^t - \h z^{t-1} \|_2^2.
\end{equation}
The additional squared-difference term allows us to absorb the lagged-gradient effect and recover a sufficient‑decrease property for LADMM; see Lemma~\ref{lem:kappa_conv}.

\begin{lemma} \label{lem:kappa_conv}
If $\rho > \max \{ \sigma_g, 1 + \lipf, \lipf ( 1 + 2 \lipf) \}$ and Assumptions \ref{assume:lb}--\ref{assume:fgrad_lip} hold, then $\{ \kappa_t \}_{t=1}^\infty$ defined in \eqref{def:kappa} is monotonically decreasing and lower bounded.
\end{lemma}
\begin{proof}
Using the definition of $\kappa_t$ in \eqref{def:kappa} together with Lemma~\ref{lem:suff_decrease_ladmm}, we obtain 
\begin{align} \label{ineq:kappa_dec}
&\kappa_{t+1} - \kappa_{t} \nonumber \\
&\hspace{1pc} = \mathcal{L}(\h x^{t+1}, \h z^{t+1}, \h v^{t+1}) - \mathcal{L}(\h x^t, \h z^t, \h v^t) + \lipf^2 \| \h z^{t+1} - \h z^{t} \|_2^2 - \lipf^2 \| \h z^t - \h z^{t-1} \|_2^2  \nonumber \\
&\hspace{1pc} \leq \left( \frac{\sigma_{g} -\rho}{2} \right) \|\h x^{t+1} -\h x^t\|_2^2 + \left( \frac{2\lipf^2 + \lipf - \rho}{2} \right) \|\h z^{t+1} - \h z^t\|_2^2 \\
& \qquad \qquad+ \left( \frac{(1 - \rho) \lipf^2}{\rho} \right) \|\h z^{t} - \h z^{t-1} \|_2^2,  
\end{align}
 where the right side must be nonpositive due to the choices of parameters. This shows that $\{ \kappa_t \}_{t=1}^\infty$  is monotonically decreasing. Next, we show that $\{ \kappa_t \}$ is lower bounded. To do so, we first observe 
\begin{align*} 
\rho \langle \h v^t, \h x^t - \h z^t \rangle &= \langle \nabla f(\h z^{t-1}), \h x^t - \h z^t \rangle \text{ by } \eqref{eq:v_update_ladmm} \nonumber \\ 
&= \langle \nabla f(\h z^{t-1}) - \nabla f(\h z^t), \h x^t - \h z^t \rangle + \langle \nabla f(\h z^t), \h x^t - \h z^t \rangle \nonumber  \\
&\geq -\frac{1}{2} \Big( \| \nabla f(\h z^{t-1}) - \nabla f(\h z^t) \|_2^2 + \| \h x^t - \h z^t \|_2^2 \Big) + \langle \nabla f(\h z^t), \h x^t - \h z^t \rangle \nonumber  \\
&\geq -\frac{\lipf^2}{2} \| \h z^t - \h z^{t-1} \|_2^2 -\frac{1}{2} \| \h x^t - \h z^t \|_2^2 + f(\h x^t ) - f(\h z^t) - \frac{\lipf}{2} \| \h x^t - \h z^t \|_2^2, \nonumber
\end{align*} 
where the last inequality is by \eqref{ineq_flip}. Combining the above with the definition of $\kappa_t$ yields
\begin{align} \label{ineq:kappa_lowerbound}
\kappa_t 
&= f(\h z^t) + g(\h x^t) + \rho \langle \h v^t, \h x^t- \h z^t \rangle + \frac{\rho}{2} \| \h x^t - \h z^t \|_2^2 + \lipf^2 \| \h z^t - \h z^{t-1} \|_2^2 \nonumber \\
&\geq f(\h x^t ) + g(\h x^t) + \frac{\lipf^2}{2} \| \h z^t - \h z^{t-1} \|_2^2 + \left( \frac{\rho - 1 - \lipf}{2} \right) \| \h x^t - \h z^t \|_2^2. 
\end{align}
Due to Assumption \ref{assume:lb}, the sum $f(\h x^t) + g(\h x^t)$ is lower bounded,  which implies $\kappa_t$ is lower bounded with the choice of parameters. This concludes the proof.
\end{proof}

Note that the condition $\rho > \max \{ \sigma_g, 1 + \lipf, \lipf ( 1 + 2 \lipf) \}$ is stronger than the ADMM condition $\rho > \max\{\sigma_g,\sqrt{2}\lipf\}$, reflecting the additional control needed to compensate for the lagged-gradient term introduced by linearization. In practice, we can estimate a lower bound of $\rho$ based on the explicit forms of both $f$ and $g$; see the experimental section and Supplementary Material for more details. 

\begin{theorem} \label{thm:linearADMM_conv}
If $\rho > \max \{ \sigma_g, 1 + \lipf, \lipf ( 1 + 2 \lipf) \}$ and Assumptions \ref{assume:coercive}--\ref{assume:fgrad_lip} hold, then the iterates $\{(\x^{t}, \z^{t}, \h v^{t})\}_{t=1}^{\infty}$ generated by \eqref{ladmm} has a subsequence convergent to a $d$-stationary point of problem \eqref{prob:fplusg}.
\end{theorem}
\begin{proof}
The proof is similar to that of Theorem \ref{thm:admm_convergence}. 
The only difference is that we use the sequence $\{\kappa_t\}_{t\ge 1}$ in place of
$\{\mathcal{L}(\h x^t,\h z^t;\h v^t)\}_{t\ge 0}$.
In particular, by telescoping summation of the inequality \eqref{ineq:kappa_dec} from $t=1$ to an arbitrary integer $T$ leads to 
\begin{align} \label{kappa_telescope}
\kappa_{T+1} - \kappa_{1} \leq &
\left( \frac{\sigma_{g} -\rho}{2} \right) \sum \limits_{t=1}^T \|\h x^{t+1} -\h x^t\|_2^2 \nonumber 
 + \left( \frac{2\lipf^2 + \lipf - \rho}{2} \right)  \sum \limits_{t=1}^T \|\h z^{t+1} - \h z^t\|_2^2 \\
 &\qquad \qquad + \left( \frac{\lipf^2 - \rho \lipf^2}{\rho} \right)  \sum \limits_{t=1}^T \|\h z^{t} - \h z^{t-1} \|_2^2,
\end{align}
which means that $\{ \kappa_t \}_{t=1}^{\infty}$ is upper bounded by $\kappa_1$. Therefore, with \eqref{ineq:kappa_lowerbound}, the sequence $\{ f(\h x^t) + g (\h x^t) \}_{t=1}^\infty$ is upper bounded. The coerciveness of $\zeta(\h x)$ in Assumption \ref{assume:coercive} guarantees that $\{ \h x^t \}_{t=1}^\infty$ is bounded, and so is $ \{ \h z^t \}_{t=1}^\infty$ by \eqref{ineq:kappa_lowerbound}. To show the boundedness of $\{\h v^t\}_{t = 1}^\infty$, we use \eqref{eq:v_update_ladmm} to obtain
\begin{align*}
\|\h v^{T+1} - \h v^2 \|_2 \leq \frac{\lipf}{\rho}\|\h z^{T} - \h z^1 \|_2, 
\end{align*}
leading to
\begin{equation*}
   \|\h v^{T+1}\|_2 \leq \|\h v^2\|_2 + \frac{\lipf}{\rho} (\|\h z^{T}\|_2+\|\h z^1\|_2),  \; \forall \, T = 1,2,\cdots.
\end{equation*}
The boundedness of  $\{\h z^t\}_{t = 1}^\infty$ guarantees the boundedness of $\{\h v^t\}_{t = 1}^\infty$. The remainder of the proof follows verbatim from that of Theorem~\ref{thm:admm_convergence}.
\end{proof}

\begin{remark}
The convergence results in Theorems~\ref{thm:admm_convergence} and~\ref{thm:linearADMM_conv}
establish subsequential convergence of the generated iterates.  Establishing an analogous full-sequence
convergence result for the present weakly convex ADMM/LADMM framework would
require additional assumptions such as the Kurdyka--{\L}ojasiewicz (KL) property \citep{AttouchBolteSvaiter2013} beyond those used here.
We therefore leave this extension for future work.
\end{remark}

\begin{remark}
The lower bounds on the penalty parameter $\rho$ in Theorems \ref{thm:admm_convergence} and \ref{thm:linearADMM_conv} are sufficient
conditions for convergence. These bounds depend on the weak convexity modulus
$\sigma_g$ and the Lipschitz constant $\lipf$, which can be estimated from
the analytical form of the objective function. In particular, the weak convexity modulus $\sigma_g$ is obtained from the regularizer or nonconvex component, while the Lipschitz constant $\lipf$ is estimated from the smooth term $f$. The corresponding estimates used in the numerical
experiments are reported in the Supplementary Material. We note that these theoretical
bounds may be conservative in practice; and convergence was observed even for values of $\rho$ below the theoretical threshold  in our experiments. Developing adaptive or problem-dependent
strategies for choosing $\rho$ is an important direction for future work.
\end{remark}


\section{Numerical Experiments}
\label{sect:numerical}
In this section, we present numerical experiments to demonstrate the efficiency of both the ADMM and LADMM methods across three case studies. All the numerical experiments were conducted in MATLAB R2021a on a machine equipped with an 11th-Gen Intel Core i7-11370H CPU and 8 GB of memory. The first two involve two-dimensional (2D) smooth functions, where we compare ADMM/LADMM  with gradient descent (GD)  \citep{demyanov1978multistep}, heavy ball (HB) \citep{polyak1964methods}, and Nesterov’s accelerated gradient (NAG) \citep{nesterov1983method}. Specifically,
the update of GD is given by:
\begin{equation}
\begin{cases}
\h p^{t+1} = -\nabla\zeta(\x^t)  \\
\h x^{t+1} = \h x^t + \alpha^{t+1}\h p^{t+1},
\end{cases}
\end{equation}
 where the stepsize $\alpha^{t+1}>0$ is updated iteratively using the Armijo-Goldstein condition \citep{armijo1966minimization}. 
 The iterations of HB method are
\begin{equation}\label{HB}
\begin{cases}
\h p^{t+1} = -\nabla\zeta(\x^t) +\beta^{t+1}\h p^t \\
\h x^{t+1} = \h x^t + \alpha^{t+1}\h p^{t+1},
\end{cases}
\end{equation}
with $\alpha^{t+1}$ adaptively chosen by the Armijo-Goldstein condition and a fixed $\beta^{t+1}=\beta$. We choose the optimal value of $\beta=0.1$ among $\{10^{-3}, 10^{-2}, 10^{-1}, 1, 10^1\}$ that gives the lowest relative errors to the ground-truth. Starting with $q^0=1$ and an initial point $\h x^0$, the NAG method is formulated as:
\begin{equation}\label{NAG}
\begin{cases}
q^{t+1} = \frac{1+\sqrt{4(q^t)^2+1}}{2} \\
\h p^{t+1} = -\nabla\zeta(\x^t)  \\
\h y^{t+1} = \h x^t + \alpha^{t+1}\h p^{t+1}\\
\h x^{t+1} = \h y^{t+1} + \frac{q^t-1}{q^{t+1}}(\h y^{t+1}-\h y^t).
\end{cases}
\end{equation}
 Similar to other gradient-based methods, the step size $\alpha^{t+1}$ in NAG is updated adaptively by the Armijo-Goldstein condition. 

Lastly, we investigate a sparse logistic regression problem in high dimensions in Section \ref{sect:log-reg}. Since the objective function is nondifferentiable, the gradient-based methods described above are not directly applicable; Instead,
in addition to ADMM and LADMM, we compare with several representative baselines, including a
majorization-minimization (MM) approach \citep{hunter2004tutorial}, proximal
linearized ADMM (PLADMM) \citep{yashtini2022convergence}, a proximal alternating linearized minimization (PALM) method \citep{bolte2014proximal}, and a standard DC algorithm
(DCA) \citep{PhamDinhLeThi1997}.

For subproblems without a closed-form solution, we use iterative inner solvers.
In the two low-dimensional test functions (Sections~\ref{3-hump} and~\ref{sec:sixhump}), the $\h x$-subproblem of ADMM is solved by gradient descent, and terminated when the gradient norm falls below $10^{-2}$ or after 8000 iterations. In the logistic regression experiment (Section \ref{sect:log-reg}), the $\h z$-subproblem of ADMM is solved by Newton's method, and terminated when $
\|\h z_{\text{new}}-\h z_{\text{old}}\|<10^{-3},
$
or after 100 iterations. Additional implementation details are provided in the Supplementary Material.

\begin{remark}
We emphasize that two different LADMM variants are used in this paper, only one of which is covered by the convergence theory in Section~\ref{sect:DC-ADMM}.
For the two low-dimensional test functions (Sections~\ref{3-hump} and~\ref{sec:sixhump}), the $\h z$-subproblem in ADMM admits a closed-form solution, while the $\h x$-subproblem does not. Accordingly, the LADMM variant implemented for these experiments linearizes $g$ in the $\h x$-subproblem, rather than $f$ in the $\h z$-subproblem as presented in Section~\ref{sec:LADMM}. This is a problem-specific deviation from the analyzed algorithms: the convergence analysis in Section~\ref{sect:DC-ADMM}, which relies on the Lipschitz continuity of $\nabla f$, does not directly apply to this variant; extending the theory to cover linearization of the weakly convex term is left for future work. By contrast, the logistic regression experiment in Section~\ref{sect:log-reg} implements the LADMM scheme exactly as analyzed in Section~\ref{sec:LADMM}, with the linearization applied to the smooth logistic loss $f$ in the $\h z$-subproblem.
\end{remark}

\subsection{Three-Hump Camel Function}\label{3-hump}
We start with a valley-shaped testing function,
 called the Three-Hump Camel Function \citep{molga2005test},  defined by,
\begin{equation}\label{three-hump}
\zeta(x_1,x_2)=2x_1^2-1.05x_1^4+\frac{x_1^6}{6}+x_1x_2+x_2^2.
\end{equation}
This function has a global minimum of $\zeta(\x^*) = 0$ at $\x^* = (0,0)$ and is weakly convex. We set $\rho=100$ for ADMM, which satisfies the sufficient condition in Theorem \ref{thm:admm_convergence}. We use the same parameter for LADMM; although our convergence analysis does not cover the chosen splitting strategy, empirical results indicate that LADMM still converges.
Further details on the weakly convex formulation, the calculation of the relevant constants, and the implementation of ADMM and LADMM for \eqref{three-hump} are provided in Supplementary Material (Appendix A). 

All competing methods start from the same initial point, chosen randomly from $[-2,2]\times[-2,2]$. The stopping criteria are identical across methods: the algorithm terminates when the gradient norm satisfies $\|\nabla\zeta(\h x)\|_2<10^{-8}$ or when the number of outer iterations reaches 15,000.

 We evaluate performance based on the decay of the objective function defined in \eqref{three-hump} and the decay of the norm of the iterates, i.e., $\|\h x^t\|_2,$ with respect to the elapsed time, measured in \textit{sec}. Since the global optimum is attained at the origin, $\|\h x^t\|_2$ quantifies the error to the optimal solution. Figure \ref{fig:3hump} presents the objective function values and the corresponding errors, demonstrating that LADMM and ADMM achieve the fastest descent while maintaining the highest accuracy among competing methods. 

  \begin{figure}[t]
     \centering 
         \includegraphics[width=0.48\textwidth]{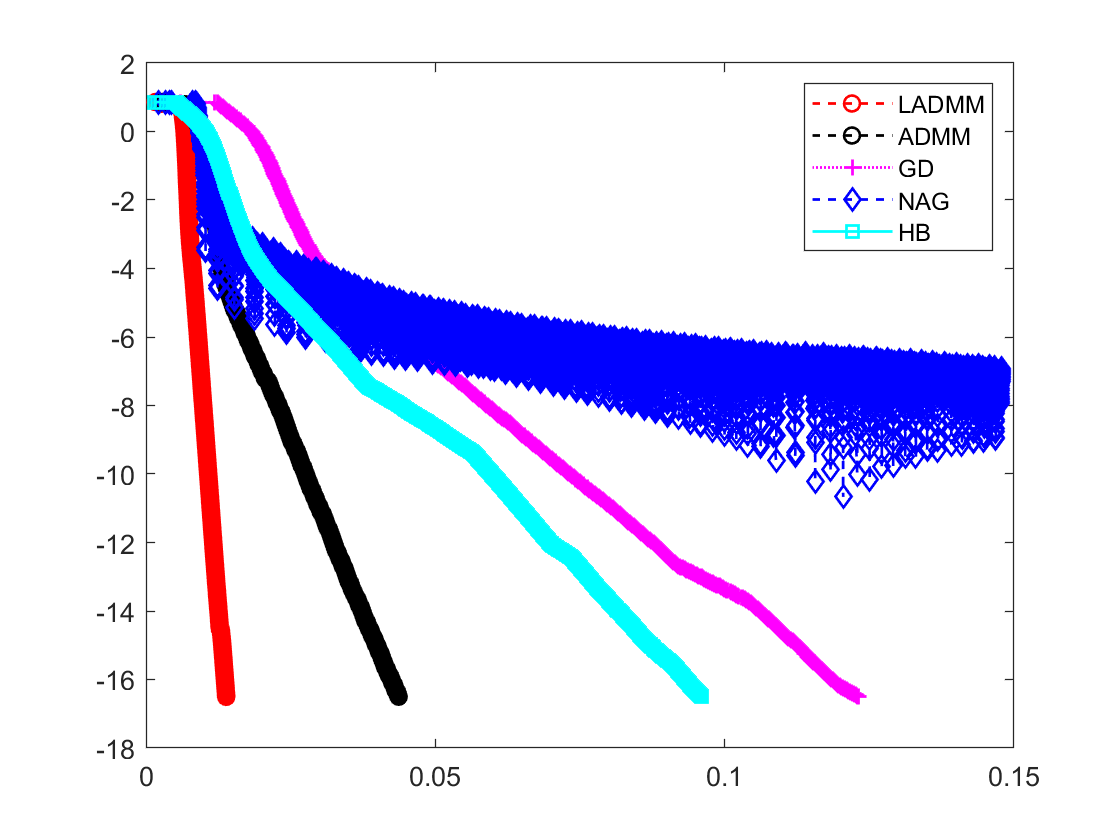}
         \includegraphics[width=0.48\textwidth]{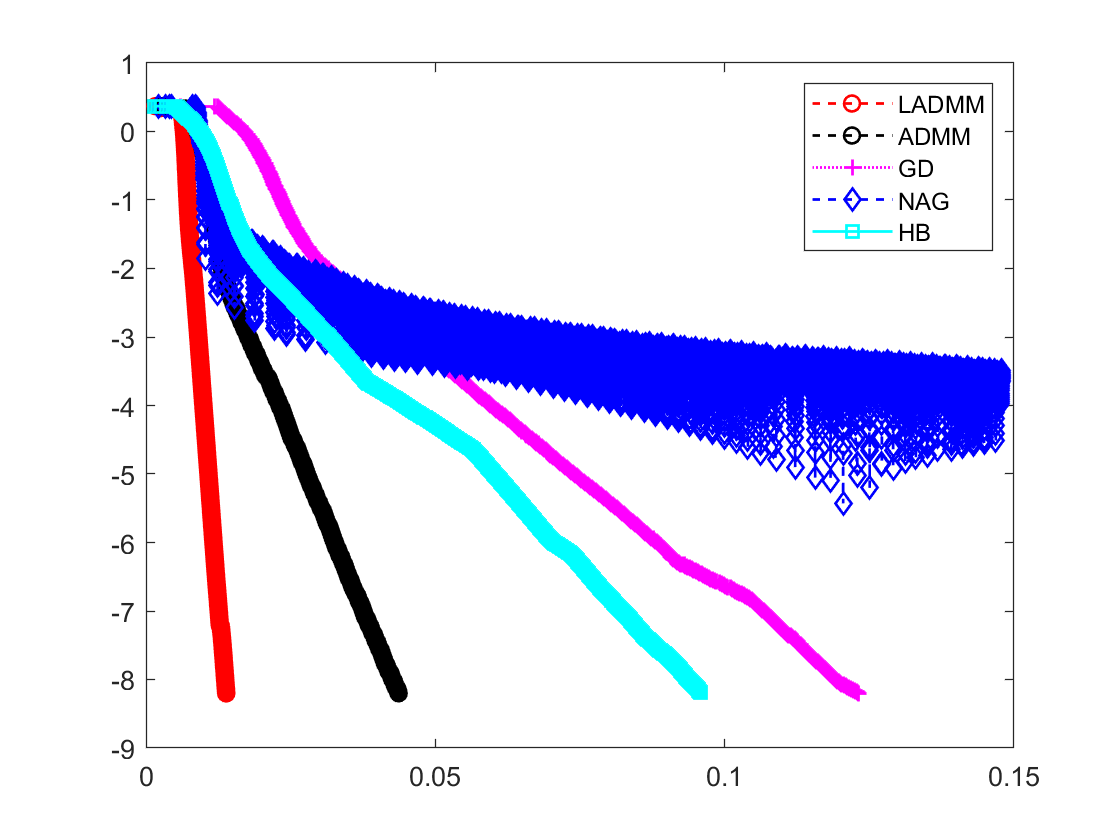}        
    \caption{Comparison of algorithms for minimizing the Three-Hump Camel Function \eqref{three-hump} in terms of the objective function (left) and 
the decay of the estimated $\|\x^t\|_2$ (right) with respect to time (\textit{sec}) on a logarithmic scale. The results are obtained from one random initialization. For ADMM and LADMM, we set $\rho=100$ and terminate the outer iterations when the gradient norm satisfies $\|\nabla\zeta(\h x)\|_2<10^{-8}$ or when the number of outer iterations reaches 15,000. LADMM and ADMM are the fastest among all competing methods.}\label{fig:3hump}
\end{figure}

We further examine computation time, the number of outer iterations, the number of inner iterations (for ADMM only), and the total number of iterations across 100 random initializations. Table \ref{tab:3hump-comparison} reports the mean values over these 100 trials, with standard deviations shown in parentheses. The total number of iterations for ADMM (including both outer and inner iterations) is nearly twice that of GD and HB; nevertheless, ADMM is faster. In particular, LADMM is the fastest method and requires fewer outer iterations than ADMM, reflecting the efficiency of its linearized subproblem updates and the absence of inner solves. NAG consistently reaches the maximum number of iterations before stopping, making it the slowest method. We note that the Armijo-based step-size strategy used for NAG may not be optimal for these nonconvex objectives; a different step-size rule could improve its performance.

\begin{table}
\centering
\small
\caption{Computation time and (inner/outer) iteration numbers for minimizing the Three-Hump Camel Function \eqref{three-hump}, highlighting LADMM and ADMM as the most efficient. Results are reported as the mean over 100 random initializations, with standard deviations shown in parentheses. Other parameter settings are the same as in Figure \ref{fig:3hump}.}

\begin{tabular}{c c c c c}
\hline
\multirow{1}{*}{Method} &  \multicolumn{1}{c}{Time (\textit{sec})} &  \multicolumn{1}{c}{Outer Iters} &  \multicolumn{1}{c}{Inner Iters} &  \multicolumn{1}{c}{Total Iters}\\   
\hline
LADMM  & 0.0086 (0.0017) & 1120.96 (103.89) & N/A& 1120.96 (103.89)\\
ADMM  & 0.0345 (0.0062) &  3417.47 (290.14)  & 17690.41 (2581.94)  & 21107.88 (2783.54)\\
GD  & 0.1005 (0.0174) & 11102.34 (985.41)& N/A               & 11102.34 (985.41)\\
NAG & 0.1305 (0.0216) & 15000 (0)       & N/A               & 15000 (0)\\
HB  & 0.0883 (0.0154) & 9984.02 (892.62)& N/A               & 9984.02 (892.62)\\
\hline
\end{tabular}
\label{tab:3hump-comparison}
\end{table}

To understand why LADMM and ADMM appear faster than other gradient-based methods, we plot the trajectories of ADMM and GD on the contour of the Three-Hump Camel Function, illustrating a case where ADMM outperforms GD in terms of speed. As shown in Figure \ref{fig:traj_3_ADMM_GD},  the trajectories of ADMM and GD are nearly identical; however, in the beginning, the points along ADMM's trajectory are more widely spaced than those of GD. This indicates that while ADMM requires inner iterations to solve the subproblem, resulting in a higher total iteration count, its alternating updates of two variables nonetheless lead to more effective minimization than standard gradient-based approaches that only involve a single variable.

\begin{figure}[t]
     \centering
         \includegraphics[width=0.48\textwidth]{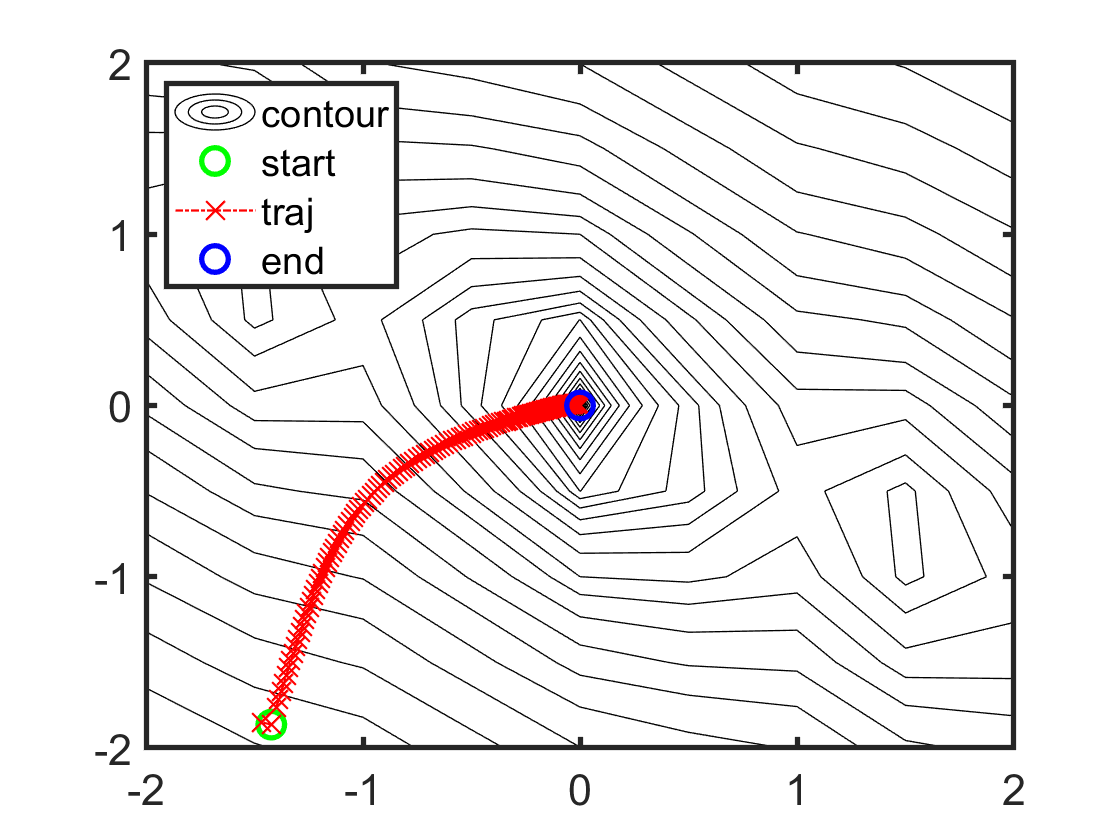}
         \includegraphics[width=0.48\textwidth]{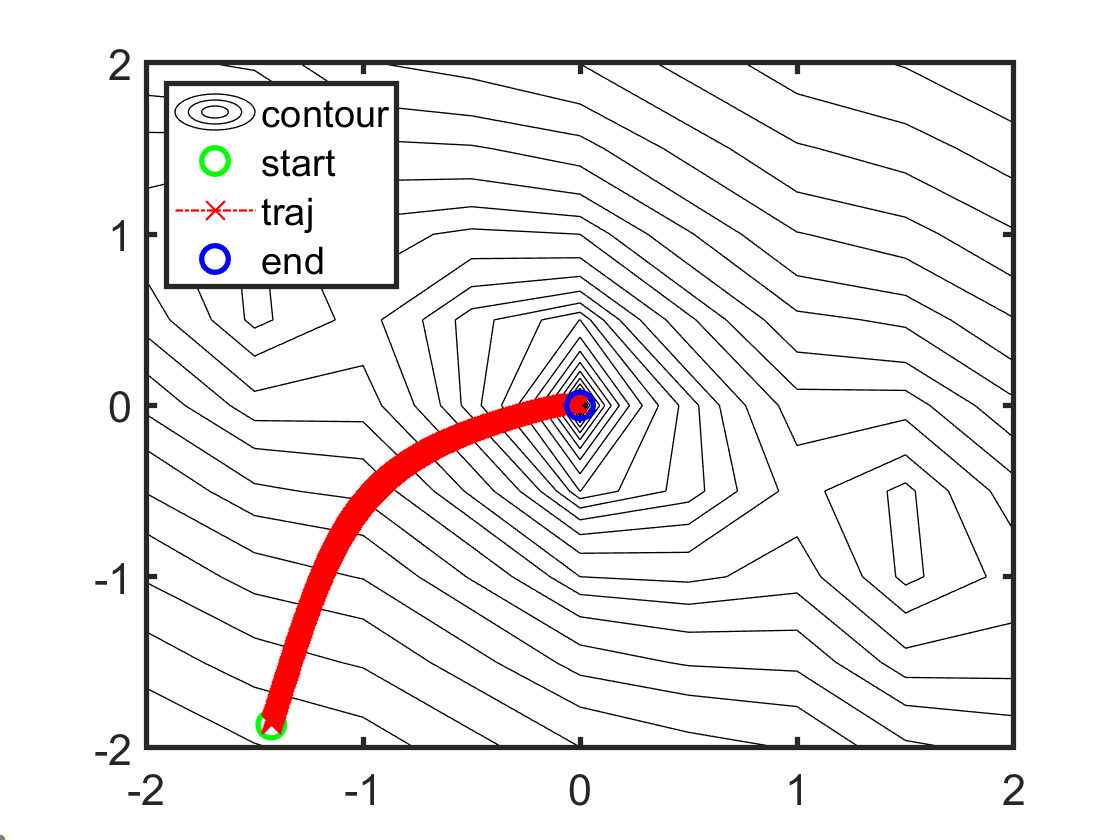} 
    \caption{Example trajectories of ADMM (left) and GD (right) starting from one particular initial point  with the same parameter settings as in Figure \ref{fig:3hump}. In the beginning, the points along ADMM's trajectory are more widely spaced than those of GD, illustrating why ADMM is faster than GD.} \label{fig:traj_3_ADMM_GD}
\end{figure}
 
\subsection{Six-Hump Camel Function}\label{sec:sixhump}

Next, we examine the Six-Hump Camel Function \citep{hedar2007global},  defined by,
\begin{equation}\label{six-hump}
	\zeta(x_1,x_2) = \left(4 - 2.1x_1^2 + \frac{x_1^4}{3} \right)x_1^2 + x_1x_2 + (-4 + 4x_2^2)x_2^2.
\end{equation}
This function possesses two symmetric global minima at approximately $(0.0898, -0.7126)$ and $(-0.0898, 0.7126)$, where the minimum function value is about $-1.0316$. 
The corresponding function $f$ for the the Six-Hump Camel is not globally Lipschitz continuous, and hence our theoretical guidelines for choosing $\rho$ are not applicable. We therefore adopt an empirical choice $\rho=100$ for both ADMM and LADMM. Despite the lack of theoretical guarantees, both methods exhibit stable convergence in our experiments. More details on the ADMM/LADMM formulation of ~\eqref{six-hump}
can be found in 
Supplementary Material (Appendix B). 


Initial points for all methods are randomly sampled from the domain $[-2, 2] \times [-2, 2]$, ensuring a diverse set of starting conditions. The stopping criteria are nearly identical to those in Section \ref{3-hump}, except that the maximum number of outer iterations is set to 1,500 to keep the method difference visible, since using 15,000 iterations (as in Section \ref{3-hump}) would make curves nearly indistinguishable.
These criteria are applied uniformly across all methods.  %

We study the convergence behavior of each method by tracking the decay of the gradient norm, $\|\nabla \zeta(\x^t)\|_2$, and the relative error, defined as $\frac{\|\h x^t - \h x^*\|_2}{\|\h x^*\|_2}$, over time. A gradient norm of $\|\nabla \zeta(\x)\|_2=0$ indicates a stationary point. The relative error quantifies the deviation between the computed solution and the true global minimum. As shown in Figure~\ref{fig:6}, the evolution of both metrics demonstrates that LADMM and ADMM converge more rapidly than the other competing methods.

We conduct a quantitative evaluation of each method's overall performance. The aggregated results, including means and standard deviations, are reported in Table~\ref{tab:6hump-comparison}. Both Tables \ref{tab:3hump-comparison} and \ref{tab:6hump-comparison} further indicate that although ADMM requires more total iterations, it achieves higher computational efficiency in these experiments. LADMM is again the fastest method, requiring fewer outer iterations than ADMM.

Figure \ref{fig:traj_6_ADMM_HB} shows the trajectories of ADMM and HB methods on the contour of the Six-Hump Camel Function, starting from a particular initial point. The widely spaced points along ADMM's trajectory, noted in Figure \ref{fig:traj_6_ADMM_HB}, are even more evident than in Figure \ref{fig:traj_3_ADMM_GD}. Both Figures \ref{fig:traj_3_ADMM_GD} and \ref{fig:traj_6_ADMM_HB} together suggest that ADMM identifies a descent direction more effectively in the early stages than HB.

 \begin{figure}[t]
     \centering
         \includegraphics[width=0.48\textwidth]{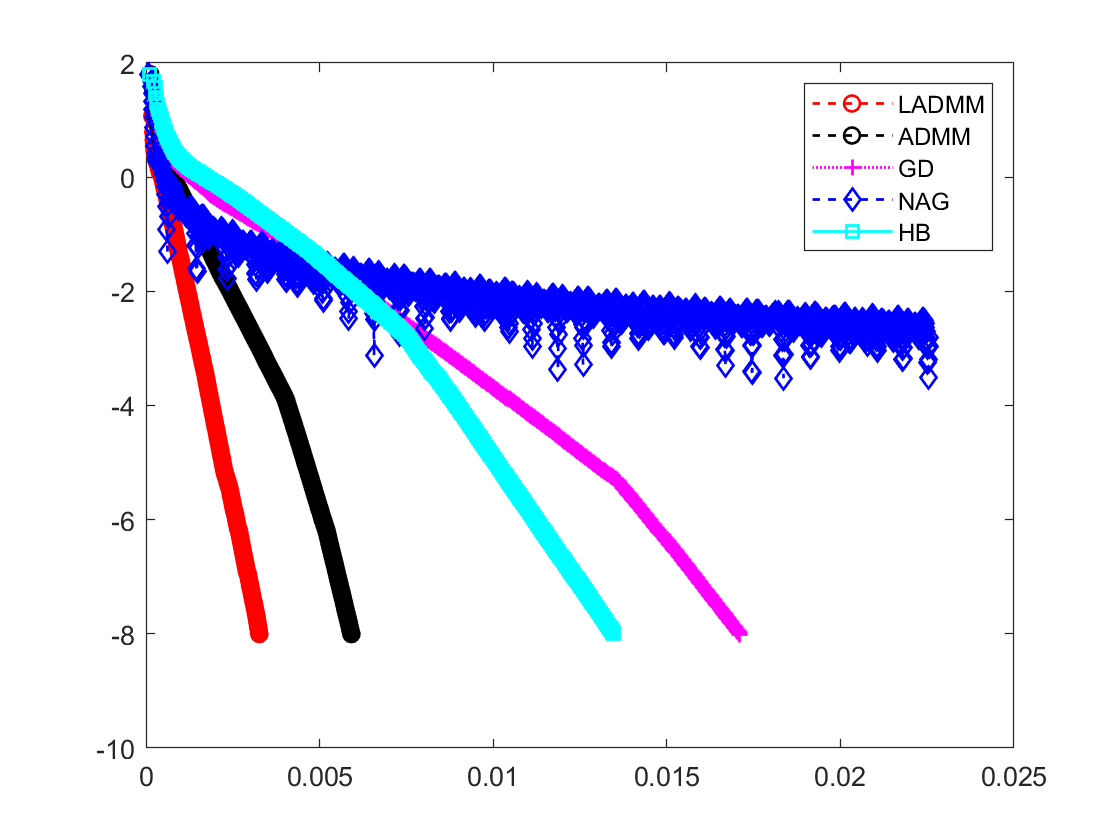}
         \includegraphics[width=0.48\textwidth]{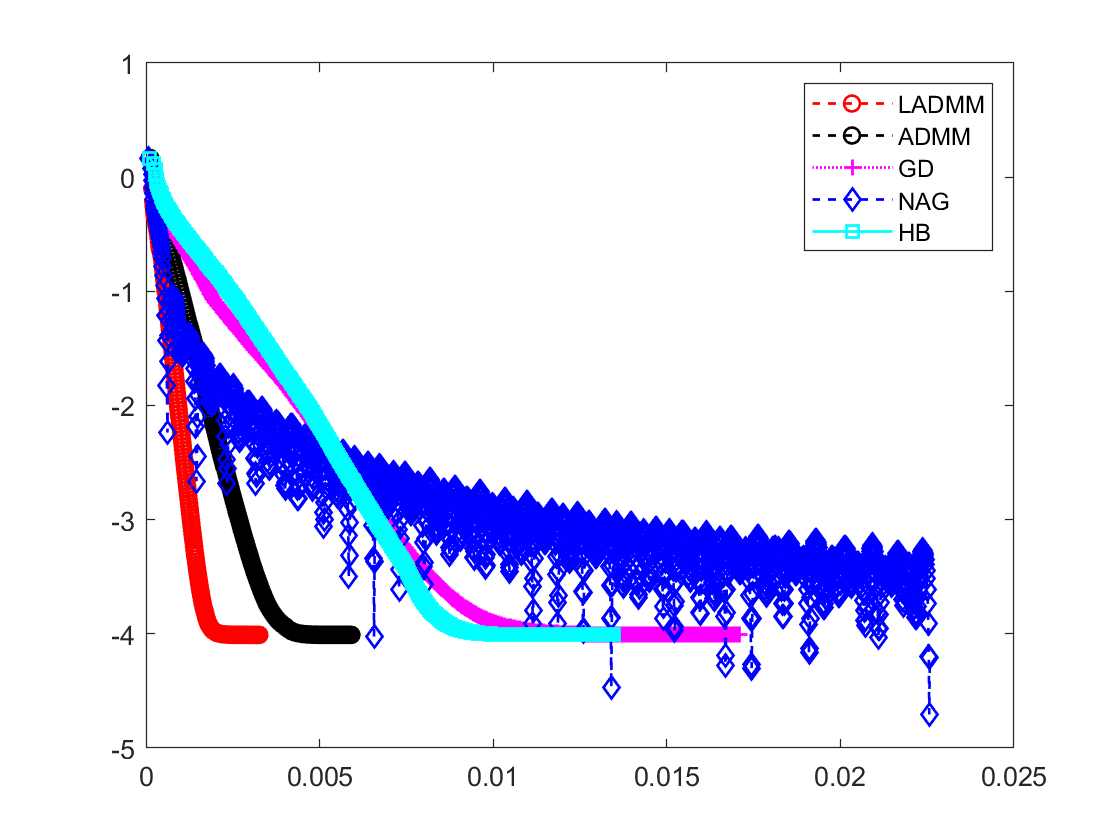}

    \caption{Comparison of algorithms for minimizing the Six-Hump Camel function \eqref{six-hump} in terms of $\|\nabla\zeta(\x)\|_2$ (left) and 
the decay of the relative error (right) with respect to time (\textit{sec}) on a logarithmic scale. The results are obtained from one random initialization. For ADMM and LADMM, we set $\rho=100$ and terminate the outer iterations when the gradient norm satisfies $\|\nabla\zeta(\h x)\|_2<10^{-8}$ or when the number of outer iterations reaches 1,500.  LADMM and ADMM are faster than other competing methods.}
    \label{fig:6}
\end{figure}

\begin{table}[htp]
\centering
\small
\caption{Computation time and (inner/outer) iteration numbers for minimizing the Six-Hump Camel function \eqref{six-hump}, highlighting LADMM and ADMM as the most efficient. Results are reported as the mean over 100 random initializations, with standard deviations shown in parentheses. Other parameter settings are the same as in Figure \ref{fig:6}.}

\begin{tabular}{c c c c c}
\hline
\multirow{1}{*}{Method} &  \multicolumn{1}{c}{Time (\textit{sec})} &  \multicolumn{1}{c}{Outer Iters} &  \multicolumn{1}{c}{Inner Iters} &  \multicolumn{1}{c}{Total Iters}\\   
\hline
LADMM  & 0.0027 (0.0010) & 224.05 (76.66) & N/A& 224.05 (76.66)\\
ADMM  & 0.0039 (0.0014) & 278.47 (91.11)  & 841.04 (290.98)  & 1119.51 (375.77)\\
GD  & 0.0108 (0.0038) & 1080.65 (353.11)& N/A               & 1080.65 (353.11)\\
NAG & 0.0141 (0.0015) & 1500 (0)       & N/A               & 1500 (0)\\
HB  & 0.0094 (0.0031) & 970.53 (318.80)& N/A               & 970.53 (318.80)\\
\hline
\end{tabular}
\label{tab:6hump-comparison}
\end{table}

\begin{figure}
     \centering
         \includegraphics[width=0.48\textwidth]{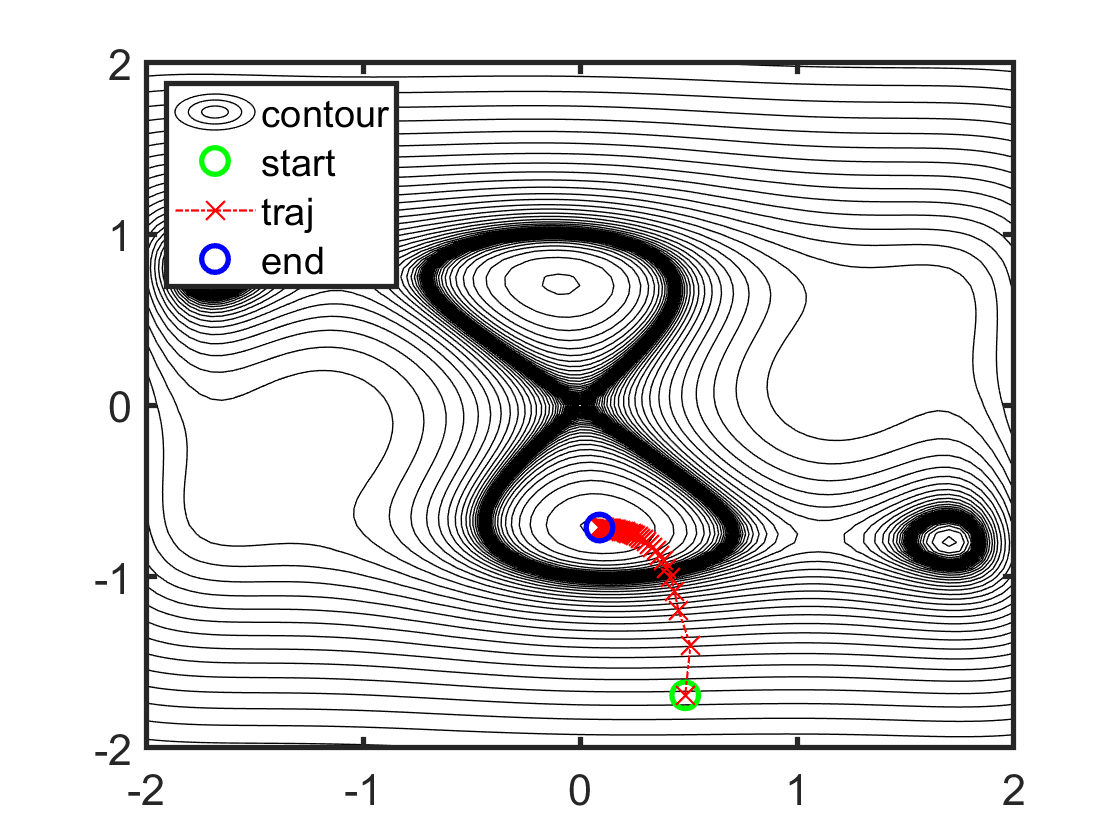}
         \includegraphics[width=0.48\textwidth]{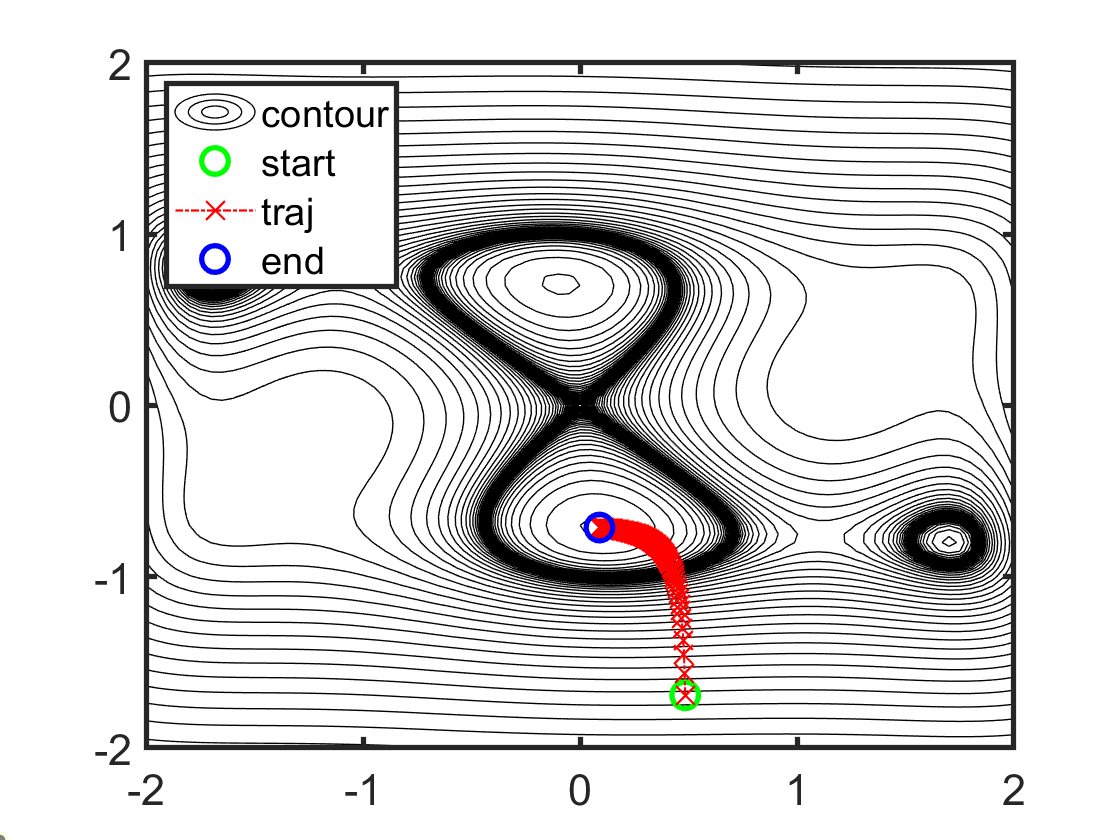} 
    \caption{Example trajectories of ADMM (left) and HB (right) with the same parameter settings as in Figure \ref{fig:6}. In the initial segment, the points along ADMM's trajectory appear more sparsely spaced than those of HB, accounting for its improved efficiency. } \label{fig:traj_6_ADMM_HB}
\end{figure}

\subsection{Logistic Regression with LOG Regularization}\label{sect:log-reg}
 
Logistic regression, which produces both a binary decision and an estimated class probability,  is a widely used classification model in statistics and machine learning, with applications in many areas including medical diagnosis \citep{hua2025clinical} and finance \citep{costaesilva2020logistic}. 
Yet in high-dimensional settings, such as gene expression analysis \citep{doustimousavi2023categorical} and large-vocabulary text classification \citep{chalkidis2020empirical}, the number of features often exceeds the number of samples, where classical logistic regression may become unstable and prone to overfitting.  Enforcing sparsity on the coefficients helps address these challenges, ensuring both statistical robustness and practical interpretability. In this section, we investigate a simulated scenario for sparse logistic regression.

We start by generating a sparse coefficient vector $\h x^*\in\mathbb{R}^n$ with a sparsity level of $10\%$, which serves as the ground-truth. In particular, for each index $j\in\{1, \cdots, n\}$, we include $j$ in the index set $S$ independently with probability 0.1. If $j\in S$, we assign $x_j^*$ a nonzero value drawn from the standard normal distribution; otherwise, we set $x_j^*=0$.
Then we generate an i.i.d. Gaussian matrix
$X_{\text{tr}}\in\mathbb{R}^{m_{\text{tr}}\times n}$ with entries
$\big(X_{\text{tr}}\big)_{ij}\stackrel{\text{iid}}{\sim}\mathcal{N}(0,1)$, where
$m_{\text{tr}}$ denotes the number of training samples and $n$ denotes the number of features. We set $(m_{\text{tr}}, n)=(200,50)$ in our experiment. We standardize the data matrix by subtracting column means. Specifically, we
compute training column means
$\mu=\frac{1}{m_{\text{tr}}}\h 1^\top X_{\text{tr}}\in\mathbb{R}^{1\times n}$ and define 
$ A_{\text{tr}}=X_{\text{tr}}-\h 1\mu$.
Then the binary labels are generated by
$$
b_{\text{tr},i}=\text{sign}\big({\h a}_{\text{tr},i}^\top \h x^*+\eta_{\text{tr},i}\big), \quad i=1,\dots,m_{\text{tr}},
$$
where $\eta_{\text{tr},i}\sim\mathcal{N}(0,\nu^2)$ acts as noise  and $\nu>0$ controls the noise level.
We set its variance $\nu^2=0.1$ in our experiments. Here, the subscript \textit{tr} indicates that the matrix $A_{\text{tr}}$ and vector $\h b_{\text{tr}}$ correspond to the training data.

Given training data $(A_{\text{tr}}, \h b_{\text{tr}}) \in \mathbb{R}^{m_{\text{tr}}\times n}\times\{-1,1\}^{m_{\text{tr}}}$ with rows $\h a_i^\top$ and labels $b_i\in\{-1,1\}$, the logistic regression function 
is defined as
\begin{align}\label{prob:logistic}
L(\h x) = \frac{1}{m_{\text{tr}}}\sum_{i=1}^{m_{\text{tr}}}\log\left(1+\exp{\left(-b_i(\h a_i^\top\h x)\right)}\right),
\end{align}
To mitigate overfitting, sparse logistic regression \citep{liu2009largescale} estimates a coefficient vector $\h x$ by solving
\begin{align} 
\min_{\h x} \ L(\h x) + \lambda R(\h x),
\end{align}
where $R(\h x)$ promotes sparsity and $\lambda > 0$
controls the strength of regularization.
In this work, we adopt a log-composite (LOG) function \citep{ke2021iteratively} defined on a univariable $x\in\mathbb R,$
\begin{align}\label{eq:LOG}
{P}_{\rm{log}}(x;\epsilon) = 
\log (\sqrt{x^2+\epsilon}+|x|),
\end{align}
with  $\epsilon>0$ to enforce sparsity, thus leading to the following 
optimization problem: 
\begin{align} \label{prob:logistic-LOG}
\min_{\h x} \ L(\h x) + \lambda\sum_{j=1}^{n}  {P}_{\rm{log}}(x_j;\epsilon).
\end{align}
Since $L(\cdot)$ is convex and the LOG penalty function is weakly convex, the problem \eqref{prob:logistic-LOG} aligns with our formulation \eqref{prob:fplusg}. 
See Supplementary Material (Appendix C) for more details on the minimization of \eqref{prob:logistic-LOG} using ADMM and LADMM.

We compare the proposed ADMM/LADMM schemes with several representative
baselines, including a
majorization-minimization (MM) method \citep{hunter2004tutorial}, a proximal
linearized ADMM method (PLADMM) \citep{yashtini2022convergence}, a proximal alternating linearized minimization (PALM) method \citep{bolte2014proximal}, and a standard DC algorithm (DCA)
\citep{PhamDinhLeThi1997}. We first describe the MM baseline in detail. In particular, we replace $P_{\rm log}(x_j;\epsilon)$ at each iteration $t$ by its concave tangent at $x_j^t$, thus leading to 
\begin{align}\label{eq:mm-sur}
\h x^{t+1}\in\argmin_{\h x}
L(\h x)+\lambda\sum_{j=1}^n w_j^t\,|x_j|,
\end{align}
where the weights are defined by
$
    w_j^t \;=\; \frac{1}{\sqrt{(x_j^t)^2+\epsilon}}.$ 
The subproblem \eqref{eq:mm-sur} is convex, which can be solved via ADMM.

For the PLADMM baseline, we use the same variable-splitting formulation as in
ADMM/LADMM and add proximal terms to the subproblems; moreover, the smooth term $L(\h x)$ is linearized, leading to a proximal linearized
ADMM update \citep{yashtini2022convergence}. For the PALM baseline, we
apply a PALM scheme \citep{bolte2014proximal} to a quadratic penalty
reformulation of the split problem. For the DCA baseline, we apply a standard DC
algorithm \citep{PhamDinhLeThi1997} based on a difference-of-convex
decomposition of the LOG regularizer, so that each DCA iteration requires the
solution of a convex subproblem.

All algorithms start from the origin and share the same stopping rule:
the algorithm is terminated when $\frac{\|\h x^{t+1}-\h x^t\|_2}{\|\h x^t\|_2} < 10^{-4}$ or when the prescribed maximum number of outer iterations is reached.

We choose the regularization parameter $\lambda$ in  \eqref{prob:logistic-LOG} by $K$-fold (we set $K=5$) cross validation on the training set $ A_{\text{tr}}$. Following a conventional choice of the set \citep{ke2021iteratively}, the candidate set $\{\lambda_i\}_{i=1}^{50}$ is generated by 
\begin{align}
\lambda_i =
\begin{cases}
10^{-5}, & i=1,\\
5\lambda_{i-1}, & 2\le i\le 6,\\
1.1\lambda_{i-1}, & 7\le i\le 44,\\
2\lambda_{i-1}, & 45\le i\le 50,
\end{cases}
\end{align}
which produces $50$ values starting at $10^{-5}$, with coarse growth early, fine spacing in the
middle, and a few larger steps at the end.
We partition the index set $\{1,\dots,m_{\text{tr}}\}$ into $K$ disjoint folds
$\{\mathcal{I}^{(k)}\}_{k=1}^K$ of (approximately) equal size.
For each $k$, define the training indices
$\mathcal{T}^{(k)}=\bigcup_{\ell\neq k}\mathcal{I}^{(\ell)}$ and the validation indices
$\mathcal{V}^{(k)}=\mathcal{I}^{(k)}$.
Given a candidate $\lambda$, let $\widehat{\h x}_{\lambda}^{(k)}$ denote the coefficient estimate obtained by fitting on $\mathcal{T}^{(k)}$. We evaluate the performance of the validation fold using the average logistic regression \citep{hastie2009elements}, i.e., by computing \eqref{prob:logistic} on the set $\mathcal{V}^{(k)}$: 
$$
L_{\text{val}}^{(k)}(\lambda)
=\frac{1}{|\mathcal{V}^{(k)}|}
\sum_{i\in \mathcal{V}^{(k)}}
\log\left(1+\exp\big(-b_{\text{tr},i} \h a_{\text{tr},i}^\top \widehat{\h x}_{\lambda}^{(k)}\big)\right),
$$
and select the optimal value of $\lambda$ that minimizes this loss,
$$
\lambda^* \in \argmin_{\lambda\in\{\lambda_i\}}
\frac{1}{K}\sum_{k=1}^K L_{\text{val}}^{(k)}(\lambda),
$$
then refit the model on the full training set using $\lambda^*$. We observe that the optimal $\lambda$ value is relatively small, so we fix $\rho=1$ and $\epsilon=0.01$ for both ADMM and LADMM in the implementation.

After selecting the optimal $\lambda$, we refit the model on the full training set.
Notably, although our analysis provides a lower bound on $\rho$ as a function of $\lambda$ (see Appendix C), we adopt the simpler choice $\rho = 1$ in all runs. This is motivated by the empirical observation that the optimal $\lambda$ values are relatively small, for which $\rho=1$ is sufficient to ensure stable convergence in practice.

We assess computational efficiency by examining the objective function value \eqref{prob:logistic-LOG} versus elapsed time. 
Figure \ref{fig:LOG} shows that LADMM descends the fastest. The computational advantage of LADMM is consistent with its closed-form $\h z$-update, which avoids the iterative inner solves required by ADMM.

\begin{figure}
     \centering
         \includegraphics[width=0.6\textwidth]{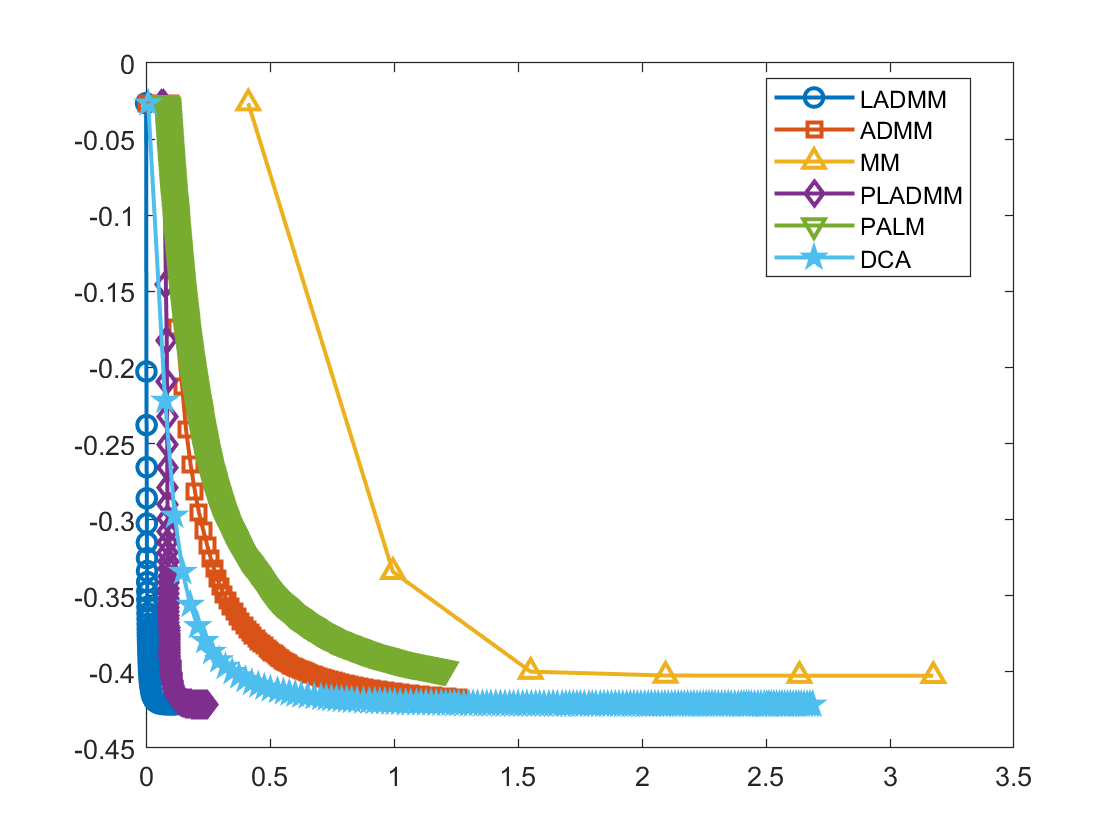} 
    \caption{Comparison of LADMM, ADMM, MM, PLADMM, PALM, and DCA for minimizing the LOG-regularized logistic regression  \eqref{prob:logistic-LOG} in terms of the objective decay with respect to time (\textit{sec}). The reported curves are obtained from one run with $m_{\text{tr}}=200$ training samples  and $n=50$ features. For ADMM and LADMM, we use $\rho=1$, $\lambda=0.0063$, and $\epsilon=0.01$. The algorithm is terminated when $\frac{\|\h x^{t+1}-\h x^t\|_2}{\|\h x^t\|_2} < 10^{-4}$ or when the prescribed maximum number of outer iterations is reached.}
    \label{fig:LOG}
\end{figure}

Then we evaluate the training performance on an independent test dataset. Consider a matrix:
$X_{\text{te}}\in\mathbb{R}^{m_{\text{te}}\times n}$ with entries
$\big(X_{\text{te}}\big)_{ij}\stackrel{\text{iid}}{\sim}\mathcal{N}(0,1)$, where $m_\text{te}$ is  the number of test samples and we set $m_\text{te}=50$.
Using the training means $\mu$, we center the data matrix \citep{kuhn2013applied}  by
$ A_{\text{te}}=X_{\text{te}}-\h 1\mu.
$
Then the true test labels are generated from the ground truth:
$$
b_{\text{te},i}=\text{sign}\big({\h a}_{\text{te},i}^\top \h x^*\big), \quad i=1,\dots,m_{\text{te}}.
$$

Suppose $\h x$ is an estimated solution, we obtain the predicted labels $\widehat{\h b}$ given by
\begin{align}
\hat b_i=\begin{cases}
+1,& (A_{\text{te}}\h{x})_i\ge 0,\\
-1,& (A_{\text{te}}\h{x})_i<0,
\end{cases}
\end{align}
for $i=1,\dots m_{\text{te}}$. Using the true labels $\h b_{\text{te}}$ and predicted values $\widehat{\h b}$,
we define the confusion matrix, i.e.,
\begin{align}
\begin{aligned}
\text{TP} &= |\{i:\hat b_i=+1, b_{\text{te},i}=+1\}|, &
\text{FP} &= |\{i: \hat b_i=+1, b_{\text{te},i}=-1\,\}|,\\
\text{TN} &= |\{i: \hat b_i=-1, b_{\text{te},i}=-1\}|, &
\text{FN} &= |\{i: \hat b_i=-1, b_{\text{te},i}=+1\,\}|,
\end{aligned}
\end{align}
based on which we calculate the following evaluation metrics: 
\begin{align}
\text{Accuracy}=\frac{\text{TP}+\text{TN}}{m_{\text{te}}},\qquad
\text{Precision}=\frac{\text{TP}}{\text{TP}+\text{FP}}\,,
\qquad
\text{Recall}=\frac{\text{TP}}{\text{TP}+\text{FN}}.
\end{align}

To further connect the numerical results with the stationarity guarantees in
Section \ref{Algorithms and Analysis}, we also report a proximal-gradient mapping residual \citep{Nesterov2013GradientMethods} at the final
iterate. For a final iterate $\h x^*$ returned by each method, we define
$$
\text{Res}_{\rm PG}(\h x^*)
=
\frac{1}{\sqrt n}
\left\|
\frac{1}{\eta}
\left[
\h x^*
-
\textbf{prox}_{\log}
\left(
\h x^*-\eta\nabla L(\h x^*);\eta\lambda
\right)
\right]
\right\|_2,
$$
where $\eta$ is the step size and the closed form expression of $\textbf{prox}_{\log}$ is provided in Supplementary Material (Appendix C). This residual measures how much the final iterate would change after one
proximal-gradient step. If the residual is zero, then one proximal-gradient step
leaves the point unchanged. Thus, smaller values of the residual indicate that the final
iterate is closer to satisfying d-stationarity. 

These metrics, together with computation time, and the proximal-gradient mapping
residual are reported in Table \ref{tab:log-comparison} as mean (standard deviation) over
$30$ independent train/test instances. 
The results show that our method is the fastest (in alignment with Figure \ref{fig:LOG}), while still achieving comparable classification performance, with accuracy, precision, and recall all around 0.97 and standard deviations in the range of 0.02 to 0.05. In addition, LADMM
attains a comparable proximal-gradient mapping residual. Among the newly
included baselines, PLADMM performs similarly to LADMM in terms of accuracy,
recall, precision, and Res$_{\rm PG}$, but has a slightly larger computational
time. PALM also remains faster than ADMM and MM in this experiment, although its
stationarity residual is larger. DCA attains slightly higher average
classification metrics, but the differences are small relative to the reported
standard deviations, and its average computational time is higher than that of
LADMM. Taken together, these findings demonstrate that LADMM achieves competitive predictive quality and comparable
stationarity residuals while delivering the lowest computational cost.

\begin{table}
\centering
\small
\caption{Computation time and performance metrics for LOG-regularized logistic regression  \eqref{prob:logistic-LOG} with $m_{\text{tr}}=200$ training samples, $n=50$ features, and $m_{\text{te}}=50$ test samples. Values are mean over 30 random instances, with standard deviations shown in parentheses. For each random instance, the regularization parameter $\lambda$ is selected by $K$-fold (we set $K=5$) cross validation over a prescribed candidate set. Other parameter settings
are the same as in Figure \ref{fig:LOG}. }

\begin{tabular}{c c c c c c }
\hline
\multirow{1}{*}{Method} &  \multicolumn{1}{c}{Time (\textit{sec})} &  
\multicolumn{1}{c}{Accuracy} &  \multicolumn{1}{c}{Recall} & 
\multicolumn{1}{c}{Precision} &  \multicolumn{1}{c}{Res$_{\rm PG}$}\\
\hline
LADMM  & 0.4577 (0.3377)  & 0.9773 (0.0227)& 0.9767 (0.0286) & 0.9772 (0.0295) & $1.8\times10^{-4}$ ($2.5\times10^{-5}$)\\
ADMM  & 1.6781 (0.8687) &  0.9733 (0.0321)   & 0.9721 (0.0361) & 0.9739 (0.0329) & $1.4\times10^{-3}$ ($8.9\times10^{-7}$)\\
MM & 2.9078 (1.7399) & 0.9667 (0.0408)              & 0.9612 (0.0474) & 0.9705 (0.0461) & $4.9\times10^{-3}$ ($2.4\times10^{-4}$)\\
PLADMM & 0.6089 (0.3942) & 0.9773 (0.0227)              & 0.9767 (0.0286) & 0.9772 (0.0295) & $2.6\times10^{-4}$ ($3.2\times10^{-5}$)\\
PALM & 1.1907 (1.0489) & 0.9667 (0.0394)              & 0.9590 (0.0457) & 0.9719 (0.0453) & $5.1\times10^{-3}$ ($1.4\times10^{-4}$)\\
DCA & 1.0601 (0.9279) & 0.9780 (0.0206)              & 0.9771 (0.0250) & 0.9776 (0.0289) & $8.2\times10^{-4}$ ($1.8\times10^{-4}$)\\
\hline
\end{tabular}
\label{tab:log-comparison}
\end{table}


To further examine the computational behavior of the methods as the problem
size increases, we perform a scaling study for
LOG-regularized logistic regression~\eqref{prob:logistic-LOG}. We set
$m=4n$ and test $n\in\{50,100,200,300\}$ and we fix $\rho=1$, $\lambda=0.0063$, and $\epsilon=0.01$.
Table \ref{tab:scale_analysis} reports the average computational time over ten
random instances. The results show that LADMM remains among the fastest methods
over the tested range. These results provide additional evidence of the
computational efficiency of LADMM.



\begin{table}[ht]
\centering
\small
\caption{Scalability comparison for LOG-regularized logistic regression \eqref{prob:logistic-LOG} with $m=4n$ and $n\in\{50,100,200,300\}$. Values are mean over 10 random instances, with standard deviations shown in parentheses. Parameters are fixed at $\rho=1$, $\lambda=0.0063$, and $\epsilon=0.01$. Other implementation settings are the same as in Figure \ref{fig:LOG}.}
{
\begin{tabular}{c c c c c}
\hline
\multirow{1}{*}{Method} 
& \multicolumn{1}{c}{$m=200,n=50$} 
& \multicolumn{1}{c}{$m=400,n=100$} 
& \multicolumn{1}{c}{$m=800,n=200$} 
& \multicolumn{1}{c}{$m=1200,n=300$} \\
\hline
LADMM  
& 0.1504 (0.0578) 
& 0.2905 (0.1109) 
& 0.5176 (0.3246) 
& 0.4622 (0.2723) \\

ADMM   
& 0.8876 (0.3122) 
& 4.6481 (0.7449) 
& 21.2790 (3.3650) 
& 33.6830 (6.2297) \\

MM     
& 2.0555 (0.8609) 
& 14.5740 (3.3491) 
& 114.3100 (25.4160) 
& 669.0600 (1169.8000) \\

PLADMM 
& 0.2125 (0.0829) 
& 0.3270 (0.1160) 
& 0.5562 (0.2171) 
& 0.5753 (0.2588) \\

PALM   
& 0.6044 (0.2329) 
& 1.3479 (0.1361) 
& 4.5345 (1.6150) 
& 7.5386 (2.1225) \\

DCA    
& 0.5177 (0.2151) 
& 1.9846 (0.1657) 
& 8.9419 (2.0311) 
& 18.4960 (3.1139) \\
\hline
\end{tabular}
}
\label{tab:scale_analysis}
\end{table}

\section{Conclusion} \label{conclusion}
We studied an optimization problem involving the sum of two functions \eqref{prob:fplusg}, where one function is convex and differentiable and the other is weakly convex (but not necessarily differentiable). This setting naturally captures a broader family of optimization models that go beyond the convex regime. Numerically, we applied both ADMM and its linearized variant to solve the problem, with the latter reducing the need for nested loops. Theoretically, we showed that both ADMM and LADMM iterates admit a subsequence converging to a directional stationary solution under mild conditions. Both convergence guarantees rely on explicit and verifiable conditions on the penalty parameter: $\rho > \max\{\sigma_g, \sqrt{2}\lipf\}$ for ADMM (Theorem~\ref{thm:admm_convergence}) and $\rho > \max\{\sigma_g, 1+\lipf, \lipf(1+2\lipf)\}$ for LADMM (Theorem~\ref{thm:linearADMM_conv}). The stronger condition for LADMM reflects the additional control needed to absorb the lagged-gradient effect introduced by linearization.
Numerical experiments on two low-dimensional smooth functions and a high-dimensional logistic regression problem demonstrate the practical efficiency of the proposed ADMM and LADMM algorithms compared with baseline methods.
In the LOG-regularized logistic regression setting, LADMM achieved
classification performance and proximal-gradient mapping residuals comparable
to the competing baselines, while
having the lowest average computational time. These results suggest that
LADMM provides an efficient strategy for this weakly convex
regularized logistic-regression model.


Some promising extensions include inexact updates that are attractive for large-scale learning problems, acceleration and preconditioning techniques within the weakly convex ADMM framework, and extending the convergence analysis to cover linearization of the weakly convex term (as employed in Sections~\ref{3-hump} and~\ref{sec:sixhump}). 



\section*{Conflict of Interest Statement}

The authors declare that the research was conducted in the absence of any commercial or financial relationships that could be construed as a potential conflict of interest.

\section*{Author Contributions} 

SM: Methodology, software, formal analysis, writing—original draft, review and editing.

\noindent
CK: Investigation, software, validation, writing—review and editing. 

\noindent
YL: Conceptualization, methodology, formal analysis, supervision, writing—review and editing.

\noindent
MA: Conceptualization, methodology, formal analysis, supervision, writing—review and editing.

\section*{Funding}
Shenghan Mei and Yifei Lou were partially supported by NSF CAREER DMS-2414705, and Chengyu Ke and Miju Ahn by NSF CRII III-1948341.




\bibliographystyle{plainnat}
\bibliography{test}

@Article{article,
	author = "Author1 LastName1 and Author2 LastName2 and Author3 LastName3",
	title = "Article Title",
	volume = "30",
	number = "30",
	pages = "10127-10134",
	year = "2013",
	doi = "10.3389/fnins.2013.12345",
	URL = "http://www.frontiersin.org/Journal/10.3389/fnins.2013.12345/abstract",
	journal = "Frontiers in Neuroscience"
}

@article{gabay1976dual,
  title={A dual algorithm for the solution of nonlinear variational problems via finite element approximation},
  author={Gabay, Daniel and Mercier, Bertrand},
  journal={Computers \& mathematics with applications},
  volume={2},
  number={1},
  pages={17--40},
  year={1976},
  publisher={Elsevier}
}

@Book{bazaraa2006nonlinear,
  author    = {Bazaraa, Mokhtar S. and Sherali, Hanif D. and Shetty, C. M.},
  title     = {Nonlinear Programming: Theory and Algorithms},
  edition   = {3},
  publisher = {John Wiley \& Sons},
  address   = {Hoboken, NJ},
  year      = {2006},
  isbn      = {978-0-471-48600-8}
}

@book{cui2021modern,
  author    = {Cui, Ying and Pang, Jong-Shi},
  title     = {Modern Nonconvex Nondifferentiable Optimization},
  series    = {MOS-SIAM Series on Optimization},
  publisher = {Society for Industrial and Applied Mathematics},
  year      = {2021},
}

@article{joki2018double,
author = {Joki, Kaisa and Bagirov, Adil M. and Karmitsa, Napsu and M\"{a}kel\"{a}, Marko M. and Taheri, Sona},
title = {Double Bundle Method for finding Clarke Stationary Points in Nonsmooth {DC} Programming},
journal = {SIAM Journal on Optimization},
volume = {28},
number = {2},
pages = {1892-1919},
year = {2018},
}

@book{clarke1983optimization,
  author={Clarke, F.H.},
  title={Optimization and Nonsmooth Optimization},
  publisher={John Wiley \& Sons},
  year={1983},
address = {London}
}

@article{pang2017computing,
  title={Computing {B}-stationary points of nonsmooth {DC} programs},
  author={Pang, Jong-Shi and Razaviyayn, Meisam and Alvarado, Alberth},
  journal={Mathematics of Operations Research},
  volume={42},
  number={1},
  pages={95--118},
  year={2017},
  publisher={INFORMS}
}

@article{boyd2011distributed,
  title   = {Distributed Optimization and Statistical Learning via the Alternating Direction Method of Multipliers},
  author  = {Boyd, Stephen and Parikh, Neal and Chu, Eric and Peleato, Borja and Eckstein, Jonathan},
  journal = {Foundations and Trends in Machine Learning},
  volume  = {3},
  number  = {1},
  pages   = {1--122},
  year    = {2011},
  publisher = {Now Publishers},
}

@article{ke2024generalized,
  author  = {Ke, Chengyu and Shin, Sunyoung and Lou, Yifei and Ahn, Miju},
  title   = {A Generalized Formulation for Group Selection via {ADMM}},
  journal = {Journal of Scientific Computing},
  year    = {2024},
  volume  = {100},
  number  = {1},
  pages   = {15}
}

@article{ke2021iteratively,
  title={Iteratively Reweighted Group Lasso Based on Log-Composite Regularization},
  author={Ke, Chengyu and Ahn, Miju and Shin, Sunyoung and Lou, Yifei},
  journal={SIAM Journal on Scientific Computing},
  volume={43},
  number={5},
  pages={S655--S678},
  year={2021},
  publisher={SIAM}
}

@article{vial1983strong,
  title={Strong and weak convexity of sets and functions},
  author={Vial, Jean-Philippe},
  journal={Mathematics of Operations Research},
  volume={8},
  number={2},
  pages={231--259},
  year={1983},
  publisher={INFORMS}
}

@article{demyanov1978multistep,
  title={A multi-step method of generalized gradient descent},
  author={Dem'yanov, V.F.},
  journal={USSR Computational Mathematics and Mathematical Physics},
  volume={18},
  number={5},
  pages={37-44},
  year={1978},
  publisher={Elsevier BV}
}

@book{rockafellar1998variational,
  title={Variational analysis},
  author={Rockafellar, R Tyrrell and Wets, Roger JB},
  year={1998},
  publisher={Springer}
}

@article{polyak1964methods,
  title={Some methods of speeding up the convergence of iteration methods},
  author={B. T. Polyak},
  journal={USSR Computational Mathematics and Mathematical Physics},
  volume={4},
  number={5},
  pages={1-17},
  year={1964},
}

@article{nesterov1983method,
  author  = {Nesterov, Yurii},
  title   = {A method of solving a convex programming problem with convergence rate {$O(1/k^2)$}},
  journal = {Soviet Mathematics Doklady},
  volume  = {27},
  pages   = {372--376},
  year    = {1983}
}

@booklet{molga2005test,
  author = {Molga, Marcin and Smutnicki, Czes{\l}aw},
  title  = {Test Functions for Optimization Needs},
  year   = {2005},
  url    = {https://robertmarks.org/Classes/ENGR5358/Papers/functions.pdf}
}

@booklet{hedar2007global,
  author = {Hedar, Abdel-Rahman},
  title  = {Global Optimization Test Problems},
  year   = {2013},
  url    = {https://www-optima.amp.i.kyoto-u.ac.jp/member/student/hedar/Hedar_files/TestGO.htm}
}

@article{armijo1966minimization,
  title   = {Minimization of functions having {L}ipschitz continuous first partial derivatives},
  author  = {Armijo, Larry},
  journal = {Pacific Journal of Mathematics},
  volume  = {16},
  number  = {1},
  pages   = {1--3},
  year    = {1966}
}

@article{hunter2004tutorial,
  author  = {Hunter, David R. and Lange, Kenneth},
  title   = {A Tutorial on {MM} Algorithms},
  journal = {The American Statistician},
  year    = {2004},
  volume  = {58},
  number  = {1},
  pages   = {30--37},
}

@article{hua2025clinical,
  author  = {Hua, Yuchen and Stead, Thor S. and George, Andrew and Ganti, Latha},
  title   = {Clinical Risk Prediction with Logistic Regression: Best Practices, Validation Techniques, and Applications in Medical Research},
  journal = {Academic Medicine \& Surgery},
  year    = {2025}
}

@article{costaesilva2020logistic,
  author  = {Costa E Silva, E. and Lopes, I. C. and Correia, A. and Faria, S.},
  title   = {A logistic regression model for consumer default risk},
  journal = {Journal of Applied Statistics},
  year    = {2020},
  volume  = {47},
  number  = {13--15},
  pages   = {2879--2894},
}

@article{doustimousavi2023categorical,
  author  = {Dousti Mousavi, N. and Aldirawi, H. and Yang, J.},
  title   = {Categorical Data Analysis for High-Dimensional Sparse Gene Expression Data},
  journal = {BioTech},
  year    = {2023},
  volume  = {12},
  number  = {3},
  pages   = {52},          
}

@article{chalkidis2020empirical,
  author    = {Chalkidis, Ilias and Fergadiotis, Manos and Kotitsas, Sotiris and
               Malakasiotis, Prodromos and Aletras, Nikolaos and Androutsopoulos, Ion},
  title     = {An Empirical Study on Large-Scale Multi-Label Text Classification Including Few and Zero-Shot Labels},
  journal   = {Proceedings of the 2020 Conference on Empirical Methods in Natural Language Processing (EMNLP)},
  year      = {2020},
  pages     = {7503--7515},
  publisher = {Association for Computational Linguistics},
}

@article{liu2009largescale,
  author    = {Liu, Jun and Chen, Jianhui and Ye, Jieping},
  title     = {Large-Scale Sparse Logistic Regression},
  journal   = {Proceedings of the 15th ACM SIGKDD International Conference on Knowledge Discovery and Data Mining},
  year      = {2009},
  pages     = {547--556},
  publisher = {Association for Computing Machinery},
  
}

@article{wang2012linearized,
  title   = {The Linearized Alternating Direction Method of Multipliers for {D}antzig Selector},
  author  = {Wang, Xiangfeng and Yuan, Xiaoming},
  journal = {SIAM Journal on Scientific Computing},
  volume  = {34},
  number  = {5},
  pages   = {A2792--A2811},
  year    = {2012}
}

@book{hastie2009elements,
  author    = {Hastie, Trevor and Tibshirani, Robert and Friedman, Jerome},
  title     = {The Elements of Statistical Learning: Data Mining, Inference, and Prediction},
  edition   = {2},
  year      = {2009},
  publisher = {Springer},
  address   = {New York, NY}
}

@book{kuhn2013applied,
  author    = {Kuhn, Max and Johnson, Kjell},
  title     = {Applied Predictive Modeling},
  year      = {2013},
  publisher = {Springer},
  address   = {New York, NY},
}

@book{boyd2004convex,
  author    = {Boyd, Stephen and Vandenberghe, Lieven},
  title     = {Convex Optimization},
  publisher = {Cambridge University Press},
  year      = {2004},
}

@article{davis2019stochastic,
  title   = {Stochastic Model-Based Minimization of Weakly Convex Functions},
  author  = {Davis, Damek and Drusvyatskiy, Dmitriy},
  journal = {SIAM Journal on Optimization},
  volume  = {29},
  number  = {1},
  pages   = {207--239},
  year    = {2019},
}

@article{shi2012projection,
  author  = {Shi, Baoli and Pang, Zhi-Feng and Yang, Yu-Fei},
  title   = {A projection method based on the splitting {B}regman iteration for the image denoising},
  journal = {Journal of Applied Mathematics and Computing},
  volume  = {39},
  pages   = {533--550},
  year    = {2012},
}

@article{gao2023convergence,
  title   = {Convergence rate analysis of an extrapolated proximal difference-of-convex algorithm},
  author  = {Gao, Lejia and Wen, Bo},
  journal = {Journal of Applied Mathematics and Computing},
  volume  = {69},
  pages   = {1403--1429},
  year    = {2023},
}

@article{liu2021firstorder,
  title   = {First-order Convergence Theory for Weakly-Convex-Weakly-Concave Min-max Problems},
  author  = {Liu, Mingrui and Rafique, Hassan and Lin, Qihang and Yang, Tianbao},
  journal = {Journal of Machine Learning Research},
  volume  = {22},
  number  = {169},
  pages   = {1--34},
  year    = {2021},
}

@article{sun2018alternating,
  author  = {Sun, Tao and Yin, Penghang and Cheng, Lizhi and Jiang, Hao},
  title   = {Alternating direction method of multipliers with difference of convex functions},
  journal = {Advanced Computational Mathematics},
  volume  = {44},
  pages   = {723--744},
  year    = {2018},
  publisher = {Springer}
}

@article{liu2019linearized,
  author  = {Liu, Qinghua and Shen, Xinyue and Gu, Yuantao},
  title   = {Linearized {ADMM} for Nonconvex Nonsmooth Optimization With Convergence Analysis},
  journal = {IEEE Access},
  volume  = {7},
  pages   = {76131--76144},
  year    = {2019},
  publisher = {IEEE}
}

@article{zhang2019fundamental,
  author  = {Zhang, Tao and Shen, Zhengwei},
  title   = {A fundamental proof of convergence of alternating direction method of multipliers for weakly convex optimization},
  journal = {Journal of Inequalities and Applications},
  year    = {2019},
  volume  = {2019},
  number  = {1},
  pages   = {1--21}
}

@article{nikolova2000local,
  author  = {Nikolova, Mila},
  title   = {Local Strong Homogeneity of a Regularized Estimator},
  journal = {SIAM Journal on Applied Mathematics},
  volume  = {61},
  number  = {2},
  pages   = {633--658},
  year    = {2000},
}

@article{lv2009unified,
  author  = {Lv, Jinchi and Fan, Yingying},
  title   = {A Unified Approach to Model Selection and Sparse Recovery Using Regularized Least Squares},
  journal = {The Annals of Statistics},
  volume  = {37},
  number  = {6A},
  pages   = {3498--3528},
  year    = {2009},
}

@article{zhang2017minimization,
  author  = {Zhang, Shuai and Xin, Jack},
  title   = {Minimization of Transformed {$\ell_1$} Penalty: Closed Form Representation and Iterative Thresholding Algorithms},
  journal = {Communications in Mathematical Sciences},
  volume  = {15},
  number  = {2},
  pages   = {511--537},
  year    = {2017},
}

@article{yashtini2022convergence,
  author = {Yashtini, Maryam},
  title = {Convergence and rate analysis of a proximal linearized {ADMM} for nonconvex nonsmooth optimization},
  journal = {Journal of Global Optimization},
  volume = {84},
  number = {4},
  pages = {913--939},
  year = {2022}
}

@article{hong2016convergence,
  author = {Hong, Mingyi and Luo, Zhi-Quan and Razaviyayn, Meisam},
  title = {Convergence analysis of alternating direction method of multipliers for a family of nonconvex problems},
  journal = {SIAM Journal on Optimization},
  volume = {26},
  number = {1},
  pages = {337--364},
  year = {2016}
}

@article{elbourkhissi2025convergence,
  author  = {El Bourkhissi, Lahcen and Necoara, Ion},
  title   = {Convergence rates for an inexact linearized {ADMM} for nonsmooth nonconvex optimization with nonlinear equality constraints},
  journal = {Computational Optimization and Applications},
  volume  = {93},
  number  = {2},
  pages   = {689--727},
  year    = {2026}
}

@article{fan2001variable,
  title={Variable selection via nonconcave penalized likelihood and its oracle properties},
  author={Fan, Jianqing and Li, Runze},
  journal={Journal of the American Statistical Association},
  volume={96},
  number={456},
  pages={1348--1360},
  year={2001}
}

@article{zhang2010nearly,
  title={Nearly unbiased variable selection under minimax concave penalty},
  author={Zhang, Cun-Hui},
  journal={The Annals of Statistics},
  volume={38},
  number={2},
  pages={894--942},
  year={2010}
}

@article{li2015global,
  author    = {Li, Guoyin and Pong, Ting Kei},
  title     = {Global Convergence of Splitting Methods for Nonconvex Composite Optimization},
  journal   = {SIAM Journal on Optimization},
  volume    = {25},
  number    = {4},
  pages     = {2434--2460},
  year      = {2015}
}

@article{bolte2014proximal,
  author    = {Bolte, J\'{e}r\^{o}me and Sabach, Shoham and Teboulle, Marc},
  title     = {Proximal Alternating Linearized Minimization for Nonconvex and Nonsmooth Problems},
  journal   = {Mathematical Programming},
  volume    = {146},
  number    = {1--2},
  pages     = {459--494},
  year      = {2014}
}

@book{nesterov2004introductory,
  author    = {Nesterov, Yurii},
  title     = {Introductory Lectures on Convex Optimization: A Basic Course},
  series    = {Applied Optimization},
  volume    = {87},
  publisher = {Kluwer Academic Publishers},
  year      = {2004}
}

@book{beck2017first,
  author    = {Beck, Amir},
  title     = {First-Order Methods in Optimization},
  series    = {MOS-SIAM Series on Optimization},
  publisher = {Society for Industrial and Applied Mathematics},
  year      = {2017}
}

@article{PhamDinhLeThi1997,
  author  = {{Pham Dinh}, Tao and {Le Thi}, Hoai An},
  title   = {Convex Analysis Approach to D.C. Programming: Theory, Algorithms and Applications},
  journal = {Acta Mathematica Vietnamica},
  volume  = {22},
  number  = {1},
  pages   = {289--355},
  year    = {1997}
}

@article{AttouchBolteSvaiter2013,
  author  = {Attouch, Hedy and Bolte, J{\'e}r{\^o}me and Svaiter, Benar Fux},
  title   = {Convergence of Descent Methods for Semi-algebraic and Tame Problems:
             Proximal Algorithms, Forward--Backward Splitting, and Regularized Gauss--Seidel Methods},
  journal = {Mathematical Programming},
  volume  = {137},
  number  = {1--2},
  pages   = {91--129},
  year    = {2013}
}

@article{Nesterov2013GradientMethods,
  author  = {Nesterov, Yurii},
  title   = {Gradient Methods for Minimizing Composite Functions},
  journal = {Mathematical Programming},
  volume  = {140},
  number  = {1},
  pages   = {125--161},
  year    = {2013}
}

\clearpage

\setcounter{algorithm}{0}
\renewcommand{\thealgorithm}{S\arabic{algorithm}}


\clearpage
\section*{Supplementary Material}
\addcontentsline{toc}{section}{Supplementary Material}

\setcounter{equation}{0}
\renewcommand{\theequation}{S\arabic{equation}}
\setcounter{algorithm}{0}
\setcounter{theorem}{0}
\setcounter{lemma}{0}
\setcounter{corollary}{0}
\setcounter{example}{0}
\setcounter{remark}{0}
\setcounter{definition}{0}

\noindent This Supplementary Material provides additional proofs and algorithmic details that are omitted from the main manuscript. Unless otherwise specified, equations, sections, and theorems (e.g., \eqref{prob:fplusg} or Section \ref{sec:wc}) refer to those in the main manuscript. Equations introduced here are numbered locally such as
\eqref{x_subproblem_wc_3_hump}, and algorithms are likewise numbered locally such as
Algorithm \ref{alg:three-hump-wc}.

\begin{appendices}

\section{Three-Hump Camel Function}\label{three-hump-wc/dc}

\subsection{ADMM for Three-Hump Camel}\label{three-hump-admm}
By setting $f(\h x)=2x_1^2+x_2^2$ and $g(\x)=\frac{x_1^6}{6}+x_1x_2-1.05x_1^4$,  we express the Three-Hump Camel Function \eqref{three-hump} into the general form of \eqref{prob:fplusg}. It is straightforward that $f(\h x)$ is convex. Here we verify that $g(\h x)$ is weakly convex, which requires to finding a constant $\sigma>0$ such that
$
    \hat{g}(\h x) = g(\h x)+\frac{\sigma}{2}\|\h x\|^2$ is convex. We compute the Hessian matrix of $\hat g$
\begin{align}
    \nabla^2 \hat{g}(\h x)=\begin{bmatrix}
        5x_1^4-12.6x_1^2+\sigma & 1 \\ 1 & \sigma
    \end{bmatrix}.
\end{align}
For any vector $[a, b]^T$,  we have
\begin{align*}
    \begin{bmatrix}
        a&b
    \end{bmatrix}\begin{bmatrix}
        5x_1^4-12.6x_1^2+\sigma & 1 \\ 1 & \sigma
    \end{bmatrix}\begin{bmatrix}
        a\\b
    \end{bmatrix}=(5x_1^4-12.6x_1^2+\sigma-1)a^2+(a+b)^2+(\sigma-1)b^2.
\end{align*}
Completing the square for $5x_1^4-12.6x_1^2+\sigma-1=5(x_1^2-1.26)^2+\sigma - 8.938$ shows that the expression is nonnegative for all $x_1$, provided $\sigma\geq 8.938$.
Therefore, $\hat{g}(\h x)$ is convex, and hence $g(\h x)$ is weakly convex.

Next, we estimate the Lipschitz constant of the function
$
f(\x)=2x_1^2+x_2^2.
$
Since
$$
\nabla f(\x)=
\begin{pmatrix}
4x_1\\
2x_2
\end{pmatrix},
\qquad
\nabla^2 f(\x)=
\begin{pmatrix}
4&0\\
0&2
\end{pmatrix},
$$
we have
$
\lipf=\|\nabla^2 f\|_2=4,
$
and hence a sufficient condition in Theorem \ref{thm:admm_convergence} for the convergence guarantee is
$
\rho>\max\{\sigma_g,\sqrt{2}\,\lipf\}
=\max\{8.938,4\sqrt2\}=8.938.$ In the experiments,
We set $\rho=100$, which satisfies this condition. 

For the special forms of $f(\cdot)$ and $g(\cdot)$, we elaborate on how to solve the $\h x$- and $\h z$-subproblems in the ADMM iteration \eqref{admm}.
Specifically, the $\h x$-subproblem in \eqref{admm} is equivalent to
\begin{align}\label{x_subproblem_wc_3_hump}
\h x^{t+1} \in \argmin_{\h x} g(\h x)+ \frac{\rho}{2} \|\h x-\h z^t+ \h v^t\|_2^2 \triangleq G_1(\x),
\end{align}
where
$   G_1(\x) = \frac{x_1^6}{6}+x_1x_2-1.05x_1^4+\frac{\rho}{2} \|\h x-\h z^t+ \h v^t\|_2^2.$
Taking the gradient of $G_1(\x)$ with respect to $\x$, we get
\begin{align}
    \nabla G_1(\x) = \begin{bmatrix}
        x_1^5+x_2-4.2x_1^3 \\x_1 \\
    \end{bmatrix}+ \rho \begin{bmatrix}
        x_1-z_1^t+v_1^t \\ x_2-z_2^t+v_2^t
    \end{bmatrix}, \nonumber
\end{align}
and hence the gradient descent update is given by
\begin{align}\label{3-hump-wc-x}
    \x_{k+1} = \x_k - \tau_1 \nabla G_1(\x_k),
\end{align}
where $\tau_1>0$ is a constant step size and the subscript $k$ denotes the inner iteration number, as opposed to superscript $t$ for outer iterations. 
We repeat the gradient descent iterations in \eqref{3-hump-wc-x} until convergence is achieved, i.e., when $\nabla G_1(\x_k)<10^{-2}$. We then set $\h x^{t+1}$ as the final iterate of $\h x_{k+1}$.

The $\h z$-subproblem in \eqref{admm} is equivalent to
\begin{align}\label{z_subproblem_wc_3_hump}
\h z^{t+1} \in \argmin_{\h z} f(\h z)+ \frac{\rho}{2} \|\h x^{t+1}-\h z+ \h v^t\|_2^2 \triangleq H_1(\z),
\end{align}
where
$ H_1(\z) = 2z_1^2+z_2^2+ \frac{\rho}{2} \|\h x^{t+1}-\h z+ \h v^t\|_2^2.$
By taking the gradient with respect to $\z$, we obtain the optimality condition:
\begin{align}
    \begin{bmatrix}
        4z_1 \\ 2z_2 \\
    \end{bmatrix} + \rho \begin{bmatrix}
        z_1-x_1^{t+1}-v_1^t \\ z_2-x_2^{t+1}-v_2^t \\
    \end{bmatrix} = \begin{bmatrix}
        0 \\0
    \end{bmatrix},\nonumber
\end{align}
which implies a closed-form solution for $\z$ given by
\begin{align}\label{3-hump-wc-z}
    \z^{t+1} = \begin{bmatrix}
        \frac{\rho (x_1^{t+1}+v_1^t)}{\rho+4} \\ \frac{\rho (x_2^{t+1}+v_2^t)}{\rho+2} \\
    \end{bmatrix}.
\end{align}
Under the stopping conditions of a maximum number of iterations and a sufficiently small gradient norm, the pseudocode for solving the Three-Hump Camel Function via ADMM is summarized in Algorithm \ref{alg:three-hump-wc}. 

\begin{algorithm}
\caption{ADMM for minimizing the Three-Hump Camel Function}
\label{alg:three-hump-wc}
\begin{algorithmic}[1]
\State Parameters: $\rho \in \mathbb{R}^+$ and tMAX, kMAX $\in \mathbb{Z}^+$; 
\State Initialize iterates $\h x^0, \h z^0, \h v^0$, and $t=0$; 
\While{$t<\text{tMAX}$ and $\|\nabla\zeta(\h x^t)\|_2>10^{-8}$} 
\State{Initialize iterates $\h x_0=\h x^t$ and $k=0$;}
\While{$k<\text{kMAX}$ and $\|\nabla(G_1(\h x_k)\|_2>10^{-2}$} 
\State{Update $\h x_{k+1}$ via \eqref{3-hump-wc-x};}
\State{$k=k+1$;}
\EndWhile
\State\Return $\h x^{t+1}=\h x_k$
\State{Update $\h z^{t+1}$ via \eqref{3-hump-wc-z};}
\State{$\h v^{t+1} = \h v^t + \h x^{t+1} -\h  z^{t+1}$;} \State{$t=t+1$;}
\EndWhile
\State \Return $\h x^*=\h x^t$
\end{algorithmic}
\end{algorithm}

\subsection{LADMM for Three-Hump Camel}\label{three-hump-ladmm}


We intend to linearize the $\h x$-subproblem in \eqref{x_subproblem_wc_3_hump} by a first-order approximation of $g(\h x)$ at $\h x^t$:
\begin{equation}\label{eq:threehump_linearize_g}
g(\h x)\approx g(\h x^t)+\langle \nabla g(\h x^t),\,\h x-\h x^t\rangle.
\end{equation}
The resulting $\h x$-subproblem becomes
\begin{align}
\h x^{t+1}\in\argmin_{\h x}\,
\langle \nabla g(\h x^t),\,\h x\rangle+\frac{\rho}{2}\|\h x-\h z^t+\h v^t\|_2^2.
\end{align}
Taking the gradient with respect to $\h x$ and setting it to zero yields the optimality condition
\begin{equation}\label{eq:threehump_x_opt}
\nabla g(\h x^t)+\rho\big(\h x^{t+1}-\h z^{t}+\h v^t\big)=\h 0,
\end{equation}
where
$$
\nabla g(\h x^t)=
\begin{bmatrix}
(x_1^t)^5+x_2^t-4.2(x_1^t)^3\\[2pt]
x_1^t
\end{bmatrix}.
$$
Solving~\eqref{eq:threehump_x_opt} gives the closed-form update
\begin{equation}\label{eq:threehump_x_update}
\h x^{t+1}=\h z^{t}-\h v^t-\frac{1}{\rho}\nabla g(\h x^t)
=
\begin{bmatrix}
z_1^{t}-v_1^t-\dfrac{(x_1^t)^5+x_2^t-4.2(x_1^t)^3}{\rho}\\[8pt]
z_2^{t}-v_2^t-\dfrac{x_1^t}{\rho}
\end{bmatrix}.
\end{equation}

With  the $\h z$-update being the same as ADMM, i.e., given by~\eqref{3-hump-wc-z}, the overall procedure is summarized in Algorithm \ref{alg:three-hump-dc}. Compared with Algorithm~\ref{alg:three-hump-wc}, no inner iteration is required. 

\begin{algorithm}[H]
\caption{LADMM for minimizing the Three-Hump Camel Function}
\label{alg:three-hump-dc}
\begin{algorithmic}[1]
\State Parameters: $\rho \in \mathbb{R}^+$ and tMAX $\in \mathbb{Z}^+$; 
\State Initialize iterates $\h x^0, \h z^0, \h v^0$, and $t=0$; 
\While{$t<\text{tMAX}$ and $\|\nabla\zeta(\h x^t)\|_2>10^{-8}$} 

\State{Update $\h x^{t+1}$ via \eqref{eq:threehump_x_update};}
\State{Update $\h z^{t+1}$ via \eqref{3-hump-wc-z};}
\State{$\h v^{t+1} = \h v^t + \h x^{t+1} -\h  z^{t+1}$;} \State{$t=t+1$;}
\EndWhile
\State\Return{$\h x^*=\h x^t$}
\end{algorithmic}
\end{algorithm}

Note that this approach (Algorithm \ref{alg:three-hump-dc}) differs to the LADMM scheme introduced in Section \ref{sec:LADMM}, where the linearization is applied on the smooth term $f$. Since the weakly convex function $g$ does not satisfy the Lipschitz‑gradient assumption, the convergence analysis developed in Section \ref{sect:DC-ADMM} is not applicable here, and we leave this extension for future work.

\section{Six-Hump Camel Function}\label{six-hump-wc/dc}
\subsection{ADMM for Six-Hump Camel}\label{six-hump-admm}
By setting $f(\h x)=4x_1^2+4x_2^4$ and $g(\x)=\frac{x_1^6}{3}+x_1x_2-2.1x_1^4-4x_2^2$, we express the Six-Hump Camel Function \eqref{six-hump} in the form of \eqref{prob:fplusg}. Clearly, $f(\h x)$ is convex. To verify the weak convexity of $g(\h x)$, it suffices to find a constant $\sigma > 0$ such that the function $\hat{g}(\h x) = g(\h x) + \frac{\sigma}{2} \| \h x \|^2$
is convex. To this end, we compute the Hessian of $\hat{g}$, leading to
\begin{align}
    \nabla^2 \hat{g}(\h x) = \begin{bmatrix}
        10x_1^4 - 25.2x_1^2 + \sigma & 1 \\ 1 & -8 + \sigma
    \end{bmatrix}.
\end{align}
For an arbitrary vector $[a, b]^\mathsf{T}$, we evaluate the associated quadratic form:
\begin{align*}
  &  \begin{bmatrix}
        a & b
    \end{bmatrix}
    \begin{bmatrix}
        10x_1^4 - 25.2x_1^2 + \sigma & 1 \\ 1 & -8 + \sigma
    \end{bmatrix}
    \begin{bmatrix}
        a \\ b
    \end{bmatrix}\\
    =& (10x_1^4 - 25.2x_1^2 + \sigma - 1)a^2 + (a + b)^2 + (\sigma - 9)b^2.
\end{align*}
We complete the square for the term $10x_1^4 - 25.2x_1^2 + \sigma - 1=10(x_1^2 - 1.26)^2 + \sigma - 15.876$, which is nonnegative for all $x_1$ provided that $\sigma \geq 15.876$. Consequently, $\hat{g}(\h x)$ is convex, and hence $g(\h x)$ is WC.

Given the specific forms of $f(\cdot)$ and $g(\cdot)$, we outline the solution of the $\h x$- and $\h z$-subproblems in the ADMM iteration \eqref{admm}.
The $\h x$-subproblem in \eqref{admm} is equivalent to
\begin{align}\label{x_subproblem_wc_6_hump}
\h x^{t+1} \in \argmin_{\h x} g(\h x)+ \frac{\rho}{2} \|\h x-\h z^t+ \h v^t\|_2^2 \triangleq G_1(\x).
\end{align}
We solve \eqref{x_subproblem_wc_6_hump} using gradient descent.
Taking the gradient of $G_1(\x)$ with respect to $\x$, we obtain
\begin{align}
    \nabla G_1(\x) = \begin{bmatrix}
        2x_1^5+x_2-8.4x_1^3 \\x_1-8x_2 \\
    \end{bmatrix}+ \rho \begin{bmatrix}
        x_1-z_1^t+v_1^t \\ x_2-z_2^t+v_2^t
    \end{bmatrix}.
\end{align}
Then, the iterative update rule becomes:
\begin{align}\label{6-hump-wc-x}
    \x_{k+1} = \x_k - \tau_1 \nabla G_1(\x_k),
\end{align}
where $\tau_1>0$ denotes the step size and $k$ indexes the inner iteration. We perform the gradient descent iterations in \eqref{6-hump-wc-x} until convergence, which is determined when the norm of the gradient satisfies $\|\nabla G_1(\x_k)\| < 10^{-2}$. The final value of $\h x_{k+1}$ is then assigned to $\h x^{t+1}$.

The $\h z$-subproblem in \eqref{admm} is given by
\begin{align}\label{z_subproblem_wc_6_hump}
\h z^{t+1} \in \argmin_{\h z} f(\h z)+ \frac{\rho}{2} \|\h x^{t+1}-\h z+ \h v^t\|_2^2 \triangleq H_1(\z).
\end{align}
Differentiating $H_1(\z)$ with respect to $\z$ yields the optimality condition: 
\begin{align}
    \begin{bmatrix}
        8z_1 \\ 16z_2^3 \\
    \end{bmatrix} + \rho \begin{bmatrix}
        z_1-x_1^{t+1}-v_1^t \\ z_2-x_2^{t+1}-v_2^t \\
    \end{bmatrix}=\begin{bmatrix}
        0 \\0
    \end{bmatrix},
\end{align}
from which we obtain the closed-form solution for $z_1=\frac{\rho(x_1^{t+1}+v_1^t)}{\rho+8}$. For $z_2$, we arrive at the following cubic equation:
\begin{align}\label{six_hump_wc_z_cubic}
    16z_2^3+\rho z_2-\rho (x_2^{t+1}+v_2^t)=0.
\end{align}
To solve this cubic equation, we analyze the discriminant
$\Delta=-(4p^3+27q^2)$, with $p = \frac{\rho}{16}$ and $q=-\frac{\rho (x_2^{t+1}+v_2^t)}{16}$, which gives rise to
the following three cases: 
\begin{itemize}
    \item Case 1: $\Delta<0$, then the cubic equation has one real root: 
\begin{align}
    z_2 = u_1^{\frac{1}{3}} + u_2^{\frac{1}{3}}, \nonumber
\end{align}
with $u_1 = -\frac{q}{2}+\sqrt{\frac{q^2}{4}+\frac{p^3}{27}}$ and $u_2=-\frac{q}{2}-\sqrt{\frac{q^2}{4}+\frac{p^3}{27}}$. Note that we ignore the two complex roots.
   \item Case 2: $\Delta=0$,  then $4p^3+27q^2=0$. If $p=0$, then both $p=q=0$, and $0$ is a triple root of the cubic equation \eqref{six_hump_wc_z_cubic}. If $p\neq0$, then the cubic equation \eqref{six_hump_wc_z_cubic} has three roots $z_{2,k}$ ($k=0,1,2$): a simple root $z_{2,0}=\frac{3q}{p}$ and a double root $z_{2,1}=z_{2,2}=-\frac{3q}{2p}$. 
   \item Case 3: $\Delta>0$, then the cubic equation \eqref{six_hump_wc_z_cubic} has three real roots, which can be expressed using the trigonometric solution:
\begin{align}
    z_{2,k} = 2\sqrt{-\frac{p}{3}}\cos\left[\arccos\left(\frac{3q}{2p}\sqrt{\frac{-3}{p}}\right)-\frac{2\pi k}{3}\right]\ \ \ \ \ \text{for} \ k=0,1,2. \nonumber
\end{align}
\end{itemize}

Since multiple solutions exist when $\Delta=0$ or $\Delta>0$, we choose the one that yields the smallest objective function value, i.e.,
\begin{align}
    z_2=\argmin_{z_{2,k}} \ H_1(z_1,z_{2,k}) \ \ \ \ \text{for} \ k=0,1,2. \nonumber
\end{align}

We summarize the update for $\h z$ as follows,
\begin{align}\label{6-hump-wc-z}
    \z^{t+1}=\begin{bmatrix}
        \frac{\rho(x_1^{t+1}+v_1^t)}{\rho+8} \\ z_2\\
    \end{bmatrix}.
\end{align}
We present the pseudo-code for solving the Six-Hump Camel Function using ADMM in Algorithm \ref{alg:six-hump-wc}.

\begin{algorithm}
\caption{ADMM for minimizing the Six-Hump Camel Function}
\label{alg:six-hump-wc}
\begin{algorithmic}[1]
\State Parameters: $\rho \in \mathbb{R}^+$ and tMAX, kMAX $\in \mathbb{Z}^+$; 
\State Initialize iterates $\h x^0, \h z^0, \h v^0$, and $t=0$; 
\While{$t<\text{tMAX}$ and $\|\nabla\zeta(\h x^t)\|_2>10^{-8}$} 
\State{Initialize iterates $\h x_0=\h x^t$ and $k=0$;}
\While{$k<\text{kMAX}$ and $\|\nabla(G_1(\h x_k)\|_2>10^{-2}$} 
\State{Update $\h x_{k+1}$ via \eqref{6-hump-wc-x};}
\State{$k=k+1$;}
\EndWhile
\State\Return $\h x^{t+1}=\h x_k$
\State{Update $\h z^{t+1}$ via \eqref{6-hump-wc-z};}
\State{$\h v^{t+1} = \h v^t + \h x^{t+1} -\h  z^{t+1}$;} \State{$t=t+1$;}
\EndWhile
\State \Return $\h x^*=\h x^t$
\end{algorithmic}
\end{algorithm}

\subsection{LADMM for Six-Hump Camel}\label{six-hump-ladmm}
As in Appendix~\ref{three-hump-ladmm} for the Three-Hump case, the ADMM $\h z$-subproblem for the Six-Hump Camel function \eqref{six-hump} admits a closed-form update (see Appendix~\ref{six-hump-admm}), whereas the $\h x$-subproblem requires an inner iterative routine. 
To reduce this cost, we again consider a problem-specific linearized ADMM variant that linearizes the $\h x$-subproblem while keeping the $\h z$-update unchanged. 
As with the Three-Hump case, we do not provide a convergence analysis here and leave it for future work.

Specifically, we linearize the $\h x$-subproblem by replacing $g(\h x)$ with its first-order approximation at $\h x^t$,
$$
g(\h x)\approx g(\h x^t)+\langle \nabla g(\h x^t),\h x-\h x^t\rangle,
$$
while keeping the $\h z$-update the same as in~\eqref{6-hump-wc-z}. 
Following the same derivation as in Appendix~\ref{three-hump-ladmm}, the resulting $\h x$-update is
\begin{equation}\label{eq:sixhump_x_generic}
\h x^{t+1}=\h z^{t}-\h v^t-\frac{1}{\rho}\nabla g(\h x^t)=
\begin{bmatrix}
z_1^{t}-v_1^t-\dfrac{2(x_1^t)^5+x_2^t-8.4(x_1^t)^3}{\rho}\\[8pt]
z_2^{t}-v_2^t-\dfrac{x_1^t-8x_2^t}{\rho}
\end{bmatrix}.
\end{equation}
Algorithm \ref{alg:six-hump-dc} outlines the LADMM procedure to solve the Six-Hump Camel Function.

\begin{algorithm}[H]
\caption{LADMM for minimizing the Six-Hump Camel Function}
\label{alg:six-hump-dc}
\begin{algorithmic}[1]
\State Parameters: $\rho \in \mathbb{R}^+$ and tMAX $\in \mathbb{Z}^+$; 
\State Initialize iterates $\h x^0, \h z^0, \h v^0$, and $t=0$; 
\While{$t<\text{tMAX}$ and $\|\nabla\zeta(\h x^t)\|_2>10^{-8}$} 

\State{Update $\h x^{t+1}$ via \eqref{eq:sixhump_x_generic};}
\State{Update $\h z^{t+1}$ via \eqref{6-hump-wc-z};}
\State{$\h v^{t+1} = \h v^t + \h x^{t+1} -\h  z^{t+1}$;} \State{$t=t+1$;}
\EndWhile
\State\Return{$\h x^*=\h x^t$}
\end{algorithmic}
\end{algorithm}

\section{Logistic Regression with LOG Regularization}
\subsection{ADMM: Logistic regression with LOG}\label{appendix log_admm}
By setting $f(\h x)=\frac{1}{m}\sum_{i=1}^m\log\left(1+\exp{\left(-b_i(\h a_i^\top\h x)\right)}\right)$ and $g(\x)=\lambda\sum_{j=1}^{n} P_{\rm{log}}(x_j;\epsilon)$ with $P_{\rm{log}}$ defined in \eqref{eq:LOG}, we express the LOG-regularized model into the general form of \eqref{prob:fplusg} in order to apply for ADMM. We start by verifying that $f(\h x)$ is convex and $g(\h x)$ is weakly convex. 

 Define a scalar function $h(s)=\log\left(1+\exp{(-s)}\right)$ and denote the sigmoid function by $l(s)=\frac{1}{1+\exp(-s)}$. We compute the first/second derivatives of $h(s)$ as follows, 
\begin{align*}
    h'(s) = -\frac{1}{1+\exp{(s)}}=-l(-s)\le 0  \text{ and }   h''(s) = l(s)(1-l(s))\ge 0.
\end{align*}
Since $h''(s)\ge 0$ for all $s\in\mathbb{R}$, the function $h(s)$ is convex.
Note that the function $f(\h x)$ is a finite average of functions of the form: $$f_i(\h x)=\log\left(1+\exp{\left(-b_i(\h a_i^\top\h x)\right)}\right)=h(b_i\h a_i^\top\h x).$$ Since $h(s)$ is convex and non-increasing, and $b_i\h a_i^\top \h x$ is an affine function in $\h x$, it follows from the composition rule of convex functions that each $f_i(\h x)=h(b_i\h a_i^\top\h x)$ is convex. Therefore, their average $f(\h x)=\frac{1}{m}\sum_{i=1}^m f_i(\h x)$ is also convex. 

Then we estimate the Liptschitz constant of $f(\h x)$. Since
$
0\le h''(s)=l(s)(1-l(s))\le \frac14, \forall s\in\mathbb{R}$ and
the Hessian of $f$ satisfies
$$
\nabla^2 f(\h x)=\frac1m\sum_{i=1}^m h''(b_i \h a_i^\top \h x)\,\h a_i \h a_i^\top
\preceq \frac1{4m}A^\top A,
$$
we have
\begin{equation}\label{Lip_f_bound}
    \lipf \le \frac{\|A\|_2^2}{4m}.
\end{equation}

Next, we verify that $g(\h x)$ is weakly convex. As the function $g(\x)=\lambda\sum_{j=1}^{n} P_{\rm{log}}(x_j;\epsilon)$ is separable with respect to each component $x_j$, the weak convexity of $g$ follows from that of the scaled penalty $P_{\rm{log}}$. In particular, \citet[Lemma 2.1]{ke2021iteratively} 
implies that for any 
$\sigma \ge \frac{2}{3\sqrt{3}\,\epsilon}$, the function
$$
\widehat{P}_{\log}(x;\epsilon) : = P_{\rm{log}}(x;\epsilon) + \frac{\sigma}{2} x ^2
$$
is convex for all $x \in \mathbb{R}$. Therefore, $P_{\log}(\cdot;\epsilon)$ is $\sigma$-weakly convex, 
and consequently $g$ is $\lambda\sigma$-weakly convex on $\mathbb{R}^n$. Together with the bound in \eqref{Lip_f_bound}, we obtain a sufficient condition for Theorem 3, namely
$$
\rho>\max\left\{\frac{2\lambda}{3\sqrt{3}\,\epsilon},
\frac{\sqrt{2}\,\|A\|_2^2}{4m}\right\}.
$$

%

For the special forms of $f(\cdot)$ and $g(\cdot)$, we elaborate on how to solve the $\h x$- and $\h z$-subproblems in the ADMM iteration \eqref{admm}. 
Specifically, the $\h x$-subproblem in \eqref{admm} is equivalent to
\begin{align}\label{x_subproblem_wc_log}
\h x^{t+1} \in \argmin_{\h x} g(\h x)+ \frac{\rho}{2} \|\h x-\h z^t+ \h v^t\|_2^2.
\end{align}

To obtain the closed-form solution of the $\h x$-subproblem under suitable conditions, we define the proximal operator for LOG by 
\begin{align}\label{proximal_operator_wc}
\textbf{prox}_{\rm{log}}(y;\mu) \in \arg\min_{x} G_y(x)\triangleq\mu P_{\rm{log}}(x) + \frac{1}{2} (x - y)^2,
\end{align}
where $\mu > 0$. By \citet[Theorem 14] {ke2024generalized}, 
if $\mu < \frac{3\sqrt{3}}{2}\,\epsilon$, then the proximal operator of LOG is given by
\begin{align}\label{LOG_prox}
\textbf{prox}_{\log}(y;\mu)=
\begin{cases}
0, & \text{if } |y|\le \dfrac{\mu}{\sqrt{\epsilon}},\\[8pt]
\text{sign}(y)\Big(\sqrt{t_1}-\big(\sqrt{t_2}+\sqrt{t_3}\big)\Big)+\dfrac{y}{2},
& \text{if } |y|> \dfrac{\mu}{\sqrt{\epsilon}},
\end{cases}
\end{align}
where 
\begin{align*}
t_1 &= \sqrt[3]{-\frac{q}{2}+\sqrt{\frac{q^2}{4}+\frac{p^3}{27}}} + \sqrt[3]{-\frac{q}{2}-\sqrt{\frac{q^2}{4}+\frac{p^3}{27}}} - \frac13\left(\frac{\epsilon}{2}-\frac{y^2}{4}\right),\\
t_2 &= \frac{-1+\sqrt{3}\,i}{2}\,\sqrt[3]{-\frac{q}{2}+\sqrt{\frac{q^2}{4}+\frac{p^3}{27}}}+\frac{-1-\sqrt{3}\,i}{2}\,\sqrt[3]{-\frac{q}{2}-\sqrt{\frac{q^2}{4}+\frac{p^3}{27}}}-\frac13\left(\frac{\epsilon}{2}-\frac{y^2}{4}\right),\\
t_3 &= \frac{-1-\sqrt{3}\,i}{2}\,\sqrt[3]{-\frac{q}{2}+\sqrt{\frac{q^2}{4}+\frac{p^3}{27}}}+\frac{-1+\sqrt{3}\,i}{2}\,\sqrt[3]{-\frac{q}{2}-\sqrt{\frac{q^2}{4}+\frac{p^3}{27}}}-\frac13\left(\frac{\epsilon}{2}-\frac{y^2}{4}\right),\\
p &= \left(\frac{\mu^2}{4}+\frac{\epsilon^2}{16}-\frac{y^2\epsilon}{8}\right)
-\frac13\left(\frac{\epsilon}{2}-\frac{y^2}{4}\right)^2,\\
q &= \frac{2}{27}\left(\frac{\epsilon}{2}-\frac{y^2}{4}\right)^3
-\frac13\left(\frac{\epsilon}{2}-\frac{y^2}{4}\right)
\left(\frac{\mu^2}{4}+\frac{\epsilon^2}{16}-\frac{y^2\epsilon}{8}\right)
-\frac{y^2\epsilon^2}{64}.
\end{align*}

Therefore, the $\h x$-subproblem admits a closed-form solution:
\begin{align}\label{log-mcp-wc-x}
\h x^*=\textbf{prox}_{\rm{log}}(\h z^t - \h v^t,\frac{\lambda}{\rho}),
\end{align}
where $\textbf{prox}_{\log}$ is applied elementwise, $\mu = \frac{\lambda}{\rho}$ and $\h y = \h z^t - \h v^t$.

The $\h z$-subproblem in \eqref{admm} is equivalent to
\begin{align}\label{logistic_z_wc}
    \h z^{t+1} \in \argmin_{\h z} f(\h z)+ \frac{\rho}{2} \|\h x^{t+1}-\h z+ \h v^t\|_2^2 \triangleq H(\z). 
\end{align}

Recall that $\h a_i\in \mathbb{R}^n$ is a feature vector and $b_i\in \{-1,1\}$ is the corresponding model, where $i=1,2,\cdots,m$. Let $\h c_i^\top=-b_i\h a_i^\top$,
then we can rewrite \eqref{logistic_z_wc} as
\begin{align}
    \h z^{t+1} \in \argmin_{\h z} \frac{1}{m}\sum_{i=1}^m\log\left(1+\exp{(\h c_i^\top\h z)}\right)+\frac{\rho}{2} \|\h x^{t+1}-\h z+\h v^t\|_2^2.
\end{align}

Thus, Newton's method can be employed to solve \eqref{logistic_z_wc}. The gradient of the function $ H(\z) $ is given by
\begin{align}
    \nabla H(\z) =
    \frac{1}{m}\sum_{i=1}^m l(\h c_i^\top\h z)\h c_i+\rho (\h z-\h x^{t+1}-\h v^t),
\end{align}
while the corresponding Hessian matrix is computed as
\begin{align}
    \nabla^2 H(\z) =
    \frac{1}{m}\sum_{i=1}^m l(\h c_i^\top\h z)\left(1-l(\h c_i^\top\h z)\right)\h c_i\h c_i^\top+\rho I_d,
\end{align}
where $l(s)=\frac{1}{1+\exp(-s)}$ denotes the sigmoid function, and $I_d$ represents the identity matrix.

Accordingly, the update rule for the variable $\h z$ takes the form
\begin{align}\label{log-wc-z}
    \z_{k+1} = \z_k - \tau \left[\nabla^2 H(\z_k)\right]^{-1}\nabla H(\h z_k),
\end{align}
where $ \tau > 0 $ is the step size and $ k $ indexes the inner iteration. The iterations of Newton's method described in \eqref{log-wc-z} are repeated until convergence is achieved. The final iterate $\h z_{k+1}$ is then set as $\h z^{t+1}$. The complete ADMM algorithm used to solve the logistic regression with LOG is summarized in Algorithm \ref{alg:log-mcp-wc}.

\begin{algorithm}[H]
\caption{ADMM for minimizing the logistic regression with LOG}
\label{alg:log-mcp-wc}
\begin{algorithmic}[1]
\State Input: data input the feature matrix $A$ and the corresponding label $\h b$; 
\State Parameters: $\rho, \lambda, \epsilon,\tau \in \mathbb{R}^+$ and tMAX, kMAX$\in \mathbb{Z}^+$; 
\State Initialize iterates $\h x^0, \h z^0, \h v^0$, and $t,k=0$; 
\While{$t<\text{tMAX}$ and $\frac{\|\h x^{t+1}-\h x^t\|_2}{\|\h x^t\|_2}>10^{-5}$} 
\State{Update $\h x^{t+1}$ via \eqref{log-mcp-wc-x};}
\While{$k<\text{kMAX}$ and $\frac{\|\h z_{k+1}-\h z_k\|_2}{\|\h z_k\|_2}>10^{-2}$}
\State{Update $\h z_{k+1}$ via \eqref{log-wc-z};}
\State{$k=k+1$;}
\EndWhile
\State\Return{$\h z^{t+1}=\h z_k$}
\State{
$\h v^{t+1} = \h v^t + \h x^{t+1} -\h  z^{t+1}$;  }  
\State{$t=t+1$;}
\EndWhile
\State \Return{$\h x^*=\h x^t$}
\end{algorithmic}
\end{algorithm}

\subsection{LADMM: Logistic regression with LOG}\label{appendix log_ladmm}
We now apply the LADMM scheme from Section \ref{sec:LADMM} to the logistic regression with LOG regularization \eqref{prob:logistic-LOG}. Using the weak convexity modulus of $g$ and the Lipschitz constant of $\nabla f$ derived in Appendix \ref{appendix log_admm}, a sufficient condition in Theorem \ref{thm:linearADMM_conv} becomes
$$
\rho>\max\left\{
\frac{2\lambda}{3\sqrt{3}\,\epsilon},\;
1+\frac{\|A\|_2^2}{4m},\;
\frac{\|A\|_2^2}{4m}
\left(1+\frac{\|A\|_2^2}{2m}\right)
\right\}.
$$
The $\h x$-update is identical to that in Appendix \ref{appendix log_admm} and is given by \eqref{log-mcp-wc-x}. 
For the $\h z$-update, applying \eqref{ladmm_z_sol} yields the explicit step
\begin{align}\label{ladmm_z_log}
    \h z^{t+1} 
    = \h x^{t+1}+\h v^t-\frac{1}{m\rho}\sum_{i=1}^m l(\h c_i^\top\h z^t)\h c_i,
\end{align}
where $l(s)=\frac{1}{1+\exp(-s)}$ denotes the sigmoid function, and $\h c_i^\top=-b_i\h a_i^\top$.
The  LADMM procedure for solving logistic regression with LOG is summarized in Algorithm \ref{alg:log-mcp-dc}.

\begin{algorithm}[H]
\caption{LADMM for minimizing the logistic regression with LOG}
\label{alg:log-mcp-dc}
\begin{algorithmic}[1]
\State Input: data input the feature matrix $A$ and the corresponding label $\h b$; 
\State Parameters: $\rho, \lambda, \epsilon \in \mathbb{R}^+$ and tMAX$\in \mathbb{Z}^+$; 
\State Initialize iterates $\h x^0, \h z^0, \h v^0$, and $t=0$; 
\While{$t<\text{tMAX}$ and $\frac{\|\h x^{t+1}-\h x^t\|_2}{\|\h x^t\|_2}>10^{-5}$} 
\State{Update $\h x^{t+1}$ via \eqref{log-mcp-wc-x};}
\State{Update $\h z^{t+1}$ via \eqref{ladmm_z_log};}
\State{
$\h v^{t+1} = \h v^t + \h x^{t+1} -\h  z^{t+1}$;  }  
\State{$t=t+1$;}
\EndWhile
\State \Return{$\h x^*=\h x^t$}
\end{algorithmic}
\end{algorithm}




\end{appendices}

\end{document}